\documentclass{amsart}

\usepackage{amsmath,amssymb,amsfonts,amsthm}
\usepackage[all]{xy}
\usepackage{enumerate}
\usepackage{tikz}
\usepackage{array}
\usepackage{mathrsfs}
\usepackage{csquotes}
\usepackage{fullpage}
\usepackage{hyperref}
\usepackage{accents}
\usepackage{mathdots}

\theoremstyle{theorem}
\newtheorem{thm}{Theorem}[section]
\newtheorem{prop}[thm]{Proposition}
\newtheorem{lem}[thm]{Lemma}
\newtheorem{coro}[thm]{Corollary}

\theoremstyle{remark}
\newtheorem{rem}[thm]{Remark}
\newtheorem{ex}[thm]{Example}

\theoremstyle{definition}
\newtheorem{defi}[thm]{Definition}

\newcommand{\Q}{\mathbb{Q}}
\newcommand{\C}{\mathbb{C}}
\newcommand{\td}[1]{\widetilde{#1}}
\newcommand{\Z}{\mathcal{Z}}
\newcommand{\T}{\mathcal{T}}

\newcommand{\dR}{\mathrm{dR}}
\newcommand{\B}{\mathrm{B}}
\renewcommand{\leq}{\leqslant}
\renewcommand{\geq}{\geqslant}
\newcommand{\eulerian}{\genfrac{\langle}{\rangle}{0pt}{}}

\title{Odd zeta motive and linear forms in odd zeta values}
\author{Cl\'{e}ment Dupont, with a joint appendix with Don Zagier}
\email{clement.dupont@umontpellier.fr \\ dbz@mpim-bonn.mpg.de}
\address{Institut Montpelli\'{e}rain Alexander Grothendieck, CNRS, Univ. Montpellier, France \;\;\; and \;\;\; 
Max-Planck-Institut f\"{u}r Mathematik\\Vivatsgasse, 7 \\53111 Bonn, Germany}

\begin{document}

\begin{abstract}
We study a family of mixed Tate motives over~$\mathbb{Z}$ whose periods are linear forms in the zeta values~$\zeta(n)$. They naturally include the Beukers--Rhin--Viola integrals for~$\zeta(2)$ and the Ball--Rivoal linear forms in odd zeta values. We give a general integral formula for the coefficients of the linear forms and a geometric interpretation of the vanishing of the coefficients of a given parity. The main underlying result is a geometric construction of a minimal ind-object in the category of mixed Tate motives over~$\mathbb{Z}$ which contains all the non-trivial extensions between simple objects. In a joint appendix with Don Zagier, we prove the compatibility between the structure of the motives considered here and the representations of their periods as sums of series.
\end{abstract}

\maketitle

\section{Introduction}

	\subsection{Constructing linear forms in zeta values}
	
		The study of the values at integers~$n\geq 2$ of the Riemann zeta function	
		$$\zeta(n)=\sum_{k\geq 1}\frac{1}{k^n}$$
		goes back to Euler, who showed that the even zeta value~$\zeta(2n)$ is a rational multiple of~$\pi^{2n}$. Lindemann's theorem thus implies that the even zeta values are transcendental numbers. It is conjectured that the odd zeta values~$\zeta(3)$,~$\zeta(5)$,~$\zeta(7),\ldots$ are algebraically independent over~$\Q[\pi]$. 

		Many of the results in the direction of this conjecture use as a key ingredient certain families of period integrals which evaluate to linear combinations of~$1$ and zeta values:
		\begin{equation}\label{eq: linear form intro}
		\int_\sigma\omega \, = \, a_0+a_2\zeta(2)+\cdots+a_n\zeta(n)\ ,
		\end{equation}
		with~$a_k\in\Q$ for every~$k$. We can cite in particular the following results (see Fischler's Bourbaki talk~\cite{fischlerbourbaki} for a more complete survey).
		\begin{enumerate}[--]
		\item Ap\'{e}ry's proof~\cite{apery} of the irrationality of~$\zeta(2)$ and~$\zeta(3)$ was simplified by Beukers~\cite{beukersapery} by using a family of integrals evaluating to linear combinations~$a_0+a_2\zeta(2)$ and~$a_0+a_3\zeta(3)$.
		\item Ball and Rivoal's proof~\cite{rivoalcras,ballrivoal} that infinitely many odd zeta values are irrational relies on a family of integrals evaluating to linear combinations (\ref{eq: linear form intro}) for which all the even coefficients~$a_2,a_4,a_6,\ldots$ vanish.
		\item Rhin and Viola's irrationality measures~\cite{rhinviolazeta2,rhinviolazeta3} for~$\zeta(2)$ and~$\zeta(3)$ are built on generalizations of the Beukers integrals and precise estimates for the coefficients~$a_2$ and~$a_3$.
		\end{enumerate}		
		In view of diophantine applications, it is crucial to have some control over the coefficients $a_k$ appearing in linear combinations (\ref{eq: linear form intro}), in particular to be able to predict the vanishing of certain coefficients.\\ 
	
		In the present article, we study the family of integrals
		\begin{equation}\label{eq: integral intro}
		\int_{[0,1]^n}\omega \;\;\;\;\; \textnormal{ with } \;\;\;\;\;\omega=\dfrac{P(x_1,\ldots,x_n)}{(1-x_1\cdots x_n)^{N}}\,dx_1\cdots dx_n \ ,
		\end{equation}
		where~$n\geq 1$ and~$N\geq 0$ are integers and~$P(x_1,\ldots,x_n)$ is a polynomial with rational coefficients. This family contains the Beukers--Rhin--Viola integrals for~$\zeta(2)$ and the Ball--Rivoal integrals. We say that an algebraic differential form~$\omega$ as in (\ref{eq: integral intro}) is \emph{integrable} if the integral in (\ref{eq: integral intro}) is absolutely convergent. Our first result is that such integrals evaluate to linear combinations of~$1$ and zeta values, with an integral formula for the coefficients.
	
		\begin{thm}\label{thm: coefficients intro}
		There exists a family~$(\sigma_2,\ldots,\sigma_n)$ of relative~$n$-cycles with rational coefficients in~$(\C^*)^n-\{x_1\cdots x_n=1\}$ such that for every integrable~$\omega$ we have
		$$\int_{[0,1]^n}\omega=a_0(\omega)+a_2(\omega)\zeta(2)+\cdots+a_n(\omega)\zeta(n)\ ,$$
		with~$a_k(\omega)$ a rational number for every~$k$, given for~$k=2,\ldots,n$ by the formula
		\begin{equation}\label{eq: coefficients a sigma intro}
		a_k(\omega)=(2\pi i)^{-k}\int_{\sigma_k}\omega \ .
		\end{equation}		
		\end{thm}
		
		The case~$n=k=2$ of this theorem is Rhin and Viola's contour formula for~$\zeta(2)$~\cite[Lemma 2.6]{rhinviolazeta2}. We note that in Theorem~\ref{thm: coefficients intro}, the relative homology classes of the~$n$-cycles~$\sigma_k$ are uniquely determined, see Theorem~\ref{thm: period matrix Z} for a precise statement. Furthermore, they are invariant, up to a sign, by the involution
		\begin{equation}\label{eq: involution intro}
		\tau:(x_1,\ldots,x_n)\mapsto (x_1^{-1},\ldots,x_n^{-1})\ ,
		\end{equation}
		which implies a general vanishing theorem for the coefficients~$a_k(\omega)$, as follows.
			
		\begin{thm}\label{thm: vanishing intro}
		For~$k=2,\ldots,n$ the relative cycle~$\tau.\sigma_k$ is homologous to~$(-1)^{k-1}\sigma_k$. Thus, for every integrable~$\omega$:
		\begin{enumerate}
		\item if~$\tau.\,\omega=\omega$ then~$a_k(\omega)=0$ for~$k\neq 0$ even;
		\item if~$\tau.\,\omega=-\omega$ then~$a_k(\omega)=0$ for~$k$ odd.
		\end{enumerate}
		\end{thm}
		
		This allows us to construct families of integrals (\ref{eq: integral intro}) which evaluate to linear combinations of~$1$ and odd zeta values, or~$1$ and even zeta values. This is the case for the integrals (see Corollary~\ref{coro: ballrivoal})
		$$\int_{[0,1]^n}\dfrac{x_1^{u_1-1}\cdots x_n^{u_n-1}(1-x_1)^{v_1-1}\cdots (1-x_n)^{v_n-1}}{(1-x_1\cdots x_n)^{N}}\,dx_1\cdots dx_n$$
		where the integers~$u_i,v_i\geq 1$ satisfy~$2u_i+v_i=N+1$ for every~$i$. Depending on the parity of the product~$(n+1)(N+1)$, the differential form is invariant or anti-invariant by~$\tau$ and we get the vanishing of even or odd coefficients. This gives a geometric interpretation of the vanishing of the coefficients in the Ball--Rivoal integrals~\cite{rivoalcras,ballrivoal}, which correspond to special values of the parameters~$u_i,v_i$.\\
		
		The fact that the vanishing of certain coefficients in the Ball--Rivoal integrals could be explained by the existence of (anti-)invariant relative cycles was suggested to me by Rivoal during a visit at Institut Fourier, Grenoble, in October 2015. The special role played by the involution~$\tau$ was first remarked by Deligne in a letter to Rivoal~\cite{deligneletterrivoal}.\\
		
		In an appendix written jointly with Don Zagier, we give an interpretation of the coefficients~$a_k(\omega)$ appearing in Theorem~\ref{thm: coefficients intro} in elementary terms, that is in terms of the natural representations of the integrals in (\ref{eq: integral intro}) as sums of series. This should be viewed as a geometric version of the dictionary between integrals and sums of series which is used in~\cite{rivoalcras,ballrivoal}. It also gives an elementary proof of the vanishing properties of Theorem~\ref{thm: vanishing intro}, which is essentially already present in the literature, see e.g.~\cite{rivoalcras,ballrivoal},~\cite[\S 8]{zudilinoddzetavalues} and~\cite[\S 3.1]{cressonfischlerrivoalseries}.\\
		
		The existence of the integral formulas (\ref{eq: coefficients a sigma intro}) follows from the computation of certain motives, which are the central objects of the present article and that we now describe.
		
	\subsection{Constructing extensions in mixed Tate motives}
	
		Recall that the category~$\mathsf{MT}(\mathbb{Z})$ of mixed Tate motives over~$\mathbb{Z}$ is a (neutral)~$\Q$-linear tannakian category defined in~\cite{delignegoncharov} and whose abstract structure is well understood. The only simple objects in~$\mathsf{MT}(\mathbb{Z})$ are the pure Tate objects~$\Q(-k)$, for~$k$ an integer, and every object in~$\mathsf{MT}(\mathbb{Z})$ has a canonical weight filtration whose graded quotients are sums of pure Tate objects. The only non-zero extension groups between the pure Tate objects are given by 
		\begin{equation}\label{eq: extensions intro}
		\mathrm{Ext}^1_{\mathsf{MT}(\mathbb{Z})}(\Q(-(2n+1)),\Q(0)) \cong \Q \hspace{1cm} (n\geq 1)\ . 
		\end{equation}
		Furthermore, a period matrix of the (essentially unique) non-trivial extension of~$\Q(-(2n+1))$ by~$\Q(0)$ has the form
		$$ \left( \begin{array}{cc}
		1 & \zeta(2n+1) \\
		0 & (2\pi i)^{2n+1}
		\end{array} \right).$$
		The difficulty of constructing linear combinations (\ref{eq: linear form intro}) with many vanishing coefficients reflects the difficulty of constructing objects of~$\mathsf{MT}(\mathbb{Z})$ with many vanishing weight-graded quotients~\cite[\S 1.4]{browndinnerparties}. In particular, the difficulty of constructing linear combinations involving only~$1$ and~$\zeta(2n+1)$ reflects the difficulty of giving a geometric construction of the extensions (\ref{eq: extensions intro}).\\
		
		In this article, we construct a minimal ind-object~$\Z^{\mathrm{odd}}$ in the category~$\mathsf{MT}(\mathbb{Z})$ which contains all the non-trivial extensions (\ref{eq: extensions intro}). The construction goes as follows. We first define, for every integer~$n$, an object~$\Z^{(n)}\in\mathsf{MT}(\mathbb{Z})$ whose periods naturally include all the integrals (\ref{eq: integral intro}). More precisely, any integrable form~$\omega$ defines a class in the de Rham realization~$\Z^{(n)}_\dR$, and the unit~$n$-cube~$[0,1]^n$ defines a class in the dual of the Betti realization~$\Z^{(n),\vee}_\B$, the pairing between these classes being the integral (\ref{eq: integral intro}). The technical heart of this article is the computation of the full period matrix of~$\Z^{(n)}$.
				
		\begin{thm}\label{thm: period matrix Z intro}
		We have a short exact sequence
		$$0\rightarrow \Q(0) \rightarrow \Z^{(n)} \rightarrow \Q(-2)\oplus\cdots\oplus\Q(-n)\rightarrow 0$$
		and~$\Z^{(n)}$ has the following period matrix which is compatible with this short exact sequence:
		$$\left( \begin{array}{ccccccc}
		1 & \zeta(2) & \zeta(3) & \cdots & \cdots & \zeta(n-1) & \zeta(n) \\
		&(2\pi i)^2 &  &  &  &  & \\
		 && (2\pi i)^3 &  &  & 0 & \\
		 &&  & \ddots &  &  & \\
		 &&  & & \ddots &  & \\
		 &0&  &  &  & (2\pi i)^{n-1} & \\
		 &&  &  &  &  & (2\pi i)^n 
		\end{array} \right)\cdot$$
		\end{thm}
		
		Concretely, this theorem says that we can find a basis~$(v_0,v_2,\ldots,v_n)$ of the de Rham realization~$\Z^{(n)}_\dR$ (which we will compute explicitly in terms of a special family of integrable forms) and a basis~$(\varphi_0,\varphi_2,\ldots,\varphi_n)$ of the dual of the Betti realization~$\Z^{(n),\vee}_\B$, such that the matrix of the integrals~$\langle \varphi_i,v_j\rangle$ is the one given. The basis element~$\varphi_0$  is the class of the unit~$n$-cube~$[0,1]^n$. Expressing the class~$[\omega]\in\Z^{(n)}_\dR$ of an integrable form~$\omega$ in the basis~$(v_0,v_2,\ldots,v_n)$ as 
		$$[\omega]=a_0(\omega)v_0+a_2(\omega)v_2+\cdots+a_n(\omega)v_n$$
		and pairing with the dual basis of the Betti realization gives the proof of Theorem~\ref{thm: coefficients intro}, with the cycles $(\sigma_2,\ldots,\sigma_n)$ chosen as representatives of the classes~$(\varphi_2,\ldots,\varphi_n)$.\\

		The involution (\ref{eq: involution intro}) plays an important role in the proof of Theorem~\ref{thm: period matrix Z intro}. It induces a natural involution, still denoted by~$\tau$, on the quotient~$\Z^{(n)}/\Q(0)\cong \Q(-2)\oplus\cdots\oplus\Q(-n)$.
		
		\begin{thm}\label{thm: sign intro}
		For~$k=2,\ldots,n$, the involution~$\tau$ acts on the direct summand~$\Q(-k)$ of~$\Z^{(n)}/\Q(0)$ by multiplication by~$(-1)^{k-1}$.
		\end{thm}
		
		This readily implies Theorem~\ref{thm: vanishing intro}. Now if we write
		$$\Z^{(n)}/\Q(0)=(\Z^{(n)}/\Q(0))_+\oplus (\Z^{(n)}/\Q(0))_-$$
		for the decomposition into invariant and anti-invariants with respect to~$\tau$ and write~$p:\Z^{(n)}\rightarrow \Z^{(n)}/\Q(0)$ for the natural projection, we may set
		$$\Z^{(n),\mathrm{odd}}:=p^{-1}((\Z^{(n)}/\Q(0))_+)$$
		whose period matrix only contains odd zeta values in the first row. The objects~$\Z^{(n),\mathrm{odd}}\in\mathsf{MT}(\mathbb{Z})$ form an inductive system, and the limit
		$$\Z^{\mathrm{odd}}:=\lim_{\stackrel{\longrightarrow}{n}}\Z^{(n),\mathrm{odd}}$$
		has an infinite period matrix
		\begin{equation}\label{eq: period matrix Z odd intro}
		\left( \begin{array}{ccccccc}
		1 & \zeta(3) & \zeta(5) & \zeta(7) & \cdots & \cdots  & \cdots\\
		&(2\pi i)^3 &  &  &  & &   \\
		 && (2\pi i)^5 &  &  & 0  & \\
		 &&  & (2\pi i)^7 &  & &  \\
		 &&  & & \ddots & &  \\
		 &0&  &  &  & \ddots & \\
		 &&&&&&\ddots
		\end{array} \right)\cdot
		\end{equation}
		
		We call~$\Z^{\mathrm{odd}}$ the \emph{odd zeta motive}. 
		
	\subsection{Related work and open questions}
	
		This article follows the program initiated by Brown~\cite{browndinnerparties}, which aims at explaining and possibly producing irrationality proofs for zeta values by means of algebraic geometry. However, the motives that we are considering are different from the general motives considered by Brown, and in particular, easier to compute. It would be interesting to determine the precise relationship between our motives and those defined in~\cite{browndinnerparties} in terms of the moduli spaces~$\mathcal{M}_{0,n+3}$. 
		
		In another direction, an explicit description of the relative cycles defined in Theorem~\ref{thm: coefficients intro} could prove helpful in proving quantitative results on the irrationality measures of zeta values, in the spirit of~\cite{rhinviolazeta2,rhinviolazeta3}.
		
	 	It is also tempting to apply our methods to other families of integrals appearing in the literature, such as the Beukers integrals for~$\zeta(3)$ and their generalizations. One should be able, for instance, to recover Rhin and Viola's contour integrals for~$\zeta(3)$~\cite[Theorem 3.1]{rhinviolazeta3}. The symmetry properties studied by Cresson, Fischler and Rivoal~\cite{cressonfischlerrivoalsymetrie} can probably be explained geometrically via finite group actions as in the present article. The ad-hoc long exact sequences appearing here should be replaced by more systematic tools such as the Orlik--Solomon bi-complexes from~\cite{dupontbiarrangements}.
				
		Finally, it should be possible to extend our results to a functional version of the periods (\ref{eq: integral intro}), where one replaces~$1-x_1\cdots x_n$ in the denominator by~$1-z\, x_1\cdots x_n$, with~$z$ a complex parameter. Such functions have already been considered in~\cite{rivoalcras, ballrivoal}. The relevant geometric objects are variations of mixed Hodge--Tate structures on~$\C-\{0,1\}$, or mixed Tate motives over~$\mathbb{A}_\Q^1-\{0,1\}$.
	
	\subsection{Contents}
	
		In \S\ref{section2} we recall some general facts about the categories in which the objects that we will be considering live, and in particular the categories~$\mathsf{MT}(\mathbb{Z})$ and~$\mathsf{MT}(\Q)$ of mixed Tate motives over~$\mathbb{Z}$ and~$\Q$. In \S\ref{section3} we introduce the zeta motives and examine their Betti and de Rham realizations. In \S\ref{section4}, which is more technical than the rest of the paper, we compute the full period matrix of the zeta motives, which allows us to define the odd zeta motives. In \S\ref{section6}, we apply our results to proving Theorems~\ref{thm: coefficients intro} and~\ref{thm: vanishing intro} on the coefficients of linear forms in zeta values.
	
	\subsection{Acknowledgements}
	
		Many thanks to Francis Brown, Pierre Cartier, Tanguy Rivoal and Don Zagier for fruitful discussions as well as comments and corrections on a preliminary version.

\section{Mixed Tate motives and their period matrices}\label{section2}

	We recall the construction of the categories~$\mathsf{MHTS}$,~$\mathsf{MT}(\Q)$ and~$\mathsf{MT}(\mathbb{Z})$, which sit as full subcategories of one another, as follows:
	$$\mathsf{MT}(\mathbb{Z}) \hookrightarrow \mathsf{MT}(\Q) \hookrightarrow \mathsf{MHTS}\ .$$

	\subsection{Mixed Hodge--Tate structures and their period matrices}
	
		\begin{defi}\label{def: mhts}
		A \emph{mixed Hodge--Tate structure} is a triple~$H=(H_\dR,H_\B,\alpha)$ consisting of:
		\begin{enumerate}[--]
		\item a finite-dimensional~$\Q$-vector space~$H_\B$, together with a finite increasing filtration indexed by even integers:~$\cdots\subset W_{2(n-1)}H_\B \subset W_{2n}H_\B \subset \cdots \subset H_\B$;
		\item a finite-dimensional~$\Q$-vector space~$H_\dR$, together with a grading indexed by even integers: $H_\dR=\bigoplus_n (H_\dR)_{2n}$;
		\item an isomorphism~$\alpha:H_\dR\otimes_\Q\C\stackrel{\simeq}{\longrightarrow} H_\B\otimes_\Q\C$;
		\end{enumerate}
		which satisfy the following conditions.
		\begin{enumerate}[--]
		\item For every integer~$n$, the isomorphism~$\alpha$ sends~$(H_\dR)_{2n}\otimes_\Q\C$ to~$W_{2n}H_\B\otimes_\Q\C$.
		\item For every integer~$n$, it induces an isomorphism~$\alpha_n:(H_\dR)_{2n}\otimes_\Q\C \stackrel{\simeq}{\longrightarrow} (W_{2n}H_\B/W_{2(n-1)}H_\B) \otimes_\Q\C\,$, which sends~$(H_\dR)_{2n}$ to~$(W_{2n}H_\B/W_{2(n-1)}H_\B)\otimes_\Q (2\pi i)^n\Q$.
		\end{enumerate}
		\end{defi}
		
		We call~$H_\B$ and~$H_\dR$ the \emph{Betti realization} and the \emph{de Rham realization} of the mixed Hodge--Tate structure, and~$\alpha$ the \emph{comparison isomorphism}. The filtration~$W$ on~$H_\B$ is called the \emph{weight filtration}. The grading on~$H_\dR$ is called the \emph{weight grading}, and the corresponding filtration~$W_{2n}H_\dR:=\bigoplus_{k\leq n}(H_\dR)_{2k}$ the \emph{weight filtration}.
		
		\begin{rem}
		More classically, a mixed Hodge--Tate structure is defined to be a mixed Hodge structure~\cite{delignehodge2,delignehodge3}~whose weight-graded quotients are of Tate type, i.e. of type~$(p,p)$ for some integer~$p$. One passes from that classical definition to Definition~\ref{def: mhts} by setting~$H_\B:=H$ and~$H_\dR:=\bigoplus_nW_{2n}H/W_{2(n-1)}H$. The isomorphism~$\alpha$ is induced by the inverses of the isomorphisms
		\begin{equation}\label{eq: splitting W}
		(W_{2n}H/W_{2(n-1)}H)\otimes_\Q\C \stackrel{\cong}{\longleftarrow} (W_{2n}H\otimes_\Q\C) \cap F^n (H\otimes_\Q\C)
		\end{equation}
		(multiplied by~$(2\pi i)^n$) which express the fact that the weight-graded quotients are of Tate type.
		\end{rem}
		
		It is convenient to view the comparison isomorphism~$\alpha:H_\dR\otimes_\Q\C\stackrel{\simeq}{\longrightarrow}H_\B\otimes_\Q\C$ as a pairing
		\begin{equation}\label{eq: pairing}
		H_\B^\vee\otimes_\Q H_\dR \longrightarrow \C \;\; , \; \varphi\otimes v\mapsto \langle\varphi,v\rangle\ ,
		\end{equation}
		where~$(\cdot)^\vee$ denotes the linear dual. The weight filtration on~$H_\B^\vee$ is defined by
		$$W_{-2n}H_\B^\vee:=(H_\B/W_{2(n-1)}H_\B)^\vee\ ,$$
		so that we have 
		$$W_{-2n}H_\B^\vee / W_{-2(n+1)}H_\B^\vee \cong (W_{2n}H_\B/W_{2(n-1)}H_\B)^\vee\ .$$ 
		The pairing (\ref{eq: pairing}) is compatible with the weight filtrations in that we have~$\langle\varphi,v\rangle=0$ for~$\varphi\in W_{-2m}H_\B^\vee$,~$v\in W_{2n}H_\dR$ and~$m<n$.\\
		
		If we choose bases for the~$\Q$-vector spaces~$H_\dR$ and~$H_\B$, then the matrix of~$\alpha$ in these bases, or equivalently the matrix of the pairing (\ref{eq: pairing}), is called a \emph{period matrix} of the mixed Hodge--Tate structure. We will always make the following assumptions on the choice of bases.
		\begin{enumerate}[--]
		\item The basis of~$H_\B$ is compatible with the weight filtration.
		\item The basis of~$H_\dR$ is compatible with the weight grading.
		\item For every~$n$, the matrix of the comparison isomorphism~$\alpha_n$ in the corresponding basis is~$(2\pi i)^n$ times the identity.
		\end{enumerate}
		This implies that any period matrix is block upper-triangular with successive blocks of~$(2\pi i)^n\,\mathrm{Id}$ on the diagonal. Conversely, any block upper-triangular matrix with successive blocks of~$(2\pi i)^n\,\mathrm{Id}$ on the diagonal is a period matrix of a mixed Hodge--Tate structure. 
		
		\begin{ex}
		Any matrix of the form
		$$ \left( \begin{array}{ccccc}
		1 & * & * & * & * \\
		0 & 2\pi i & 0 & * & *\\
		0 & 0 & 2\pi i & * & *\\
		0 & 0 & 0 & (2\pi i)^2 & 0 \\
		0 & 0 & 0 & 0 & (2\pi i)^2 
		\end{array} \right)$$ 
		defines a mixed Hodge--Tate structure~$H$ such that~$H_\dR=(H_\dR)_0\oplus (H_\dR)_2\oplus (H_\dR)_4$ has graded dimension~$(1,2,2)$.
		\end{ex}
		
	\subsection{The category of mixed Hodge--Tate structures}
		
		We denote by~$\mathsf{MHTS}$ the category of mixed Hodge--Tate structures. It is a neutral tannakian category over~$\Q$, which means in particular that it is an abelian~$\Q$-linear category equipped with a~$\Q$-linear tensor product~$\otimes$. We note that an object~$H\in\mathsf{MHTS}$ is endowed with a canonical weight filtration~$W$ by subobjects such that the morphisms in~$\mathsf{MHTS}$ are strictly compatible with~$W$. We have two natural fiber functors
		\begin{equation}\label{eq: realization MHTS}
		\omega_\B:\mathsf{MHTS}\rightarrow \mathsf{Vect}_\Q \hspace{1cm} \textnormal{and} \hspace{1cm} \omega_\dR:\mathsf{MHTS}\rightarrow \mathsf{Vect}_\Q
		\end{equation}
		from~$\mathsf{MHTS}$ to the category of finite-dimensional vector spaces over~$\mathbb{Q}$, which only remember the Betti realization~$H_\B$ and the de Rham realization~$H_\dR$ respectively. We note that the de Rham realization functor~$\omega_\dR$ factors through the category of finite-dimensional graded vector spaces. The comparison isomorphisms~$\alpha$ gives an isomorphism between the complexifications of the two fiber functors:
		\begin{equation}\label{eq: comparison MHTS}
		\omega_\dR\otimes_\Q\C \stackrel{\simeq}{\longrightarrow} \omega_\B\otimes_\Q\C\ .
		\end{equation}
		
		For an integer~$n$, we denote by~$\Q(-n)$ the mixed Hodge--Tate structure whose period matrix is the~$1\times 1$ matrix~$\left( (2\pi i)^n\right)$. Its weight grading and filtration are concentrated in weight~$2n$, hence we call it the \emph{pure Tate structure} of weight~$2n$. For~$H$ a mixed Hodge--Tate structure, the tensor product~$H\otimes\Q(-n)$ is simply denoted by~$H(-n)$ and called the~$n$-th \emph{Tate twist} of~$H$. A period matrix of~$H(-n)$ is obtained by multiplying a period matrix of~$H$ by~$(2\pi i)^n$. The weight grading and filtration of~$H(-n)$ are those of~$H$, shifted by~$2n$. 
		
	\subsection{Extensions between pure Tate structures}
		
		The pure Tate structures~$\Q(-n)$ are the only simple objects of the category~$\mathsf{MHTS}$. The extensions between them are easily described. Up to a Tate twist, it is enough to describe the extensions of~$\Q(-n)$ by~$\Q(0)$ for some integer~$n$. The corresponding extension group is given by
		$$\mathrm{Ext}_{\mathsf{MHTS}}^1(\Q(-n),\Q(0)) = \begin{cases} \C / (2\pi i)^{n}\Q & \textnormal{ if } n>0; \\ 0 & \textnormal{ otherwise.}\end{cases}$$
		More concretely, the extension corresponding to a number~$z\in \C / (2\pi i)^{n}\Q\,$ has a period matrix
		$$ \left( \begin{array}{cc}
		1 & z \\
		0 & (2\pi i)^{n}
		\end{array} \right).$$
		We note that the higher extension groups vanish:~$\mathrm{Ext}^r_{\mathsf{MHTS}}(H,H')=0$ for~$r\geq 2$ and~$H$,~$H'$ two mixed Hodge--Tate structures.
		
		\begin{ex}
		For a complex number~$a\in\C-\{0,1\}$, the cohomology group~$H^1(\C^*,\{1,a\})$ is an extension of~$\Q(-1)$ by~$\Q(0)$ corresponding to~$z=\log(a)\in\C/(2\pi i)\Q$. It is called the \emph{Kummer extension} of parameter~$a$.
		\end{ex}
		
	\subsection{Mixed Tate motives over~$\mathbb{Q}$}\label{par: def MT}
	
		Let~$\mathsf{DM}(\mathbb{Q})$ denote Voevodsky's triangulated category of motives over~$\mathbb{Q}$~\cite{voevodskytriangulated}. It is a~$\Q$-linear triangulated tensor category whose objects can be described in terms of complexes of varieties and whose morphisms can be described in terms of algebraic cycles (in particular, in terms of Bloch's higher Chow groups). There are invertible objects~$\Q(-n)\in\mathsf{DM}(\Q)$, where~$\Q(-1)$ is the reduced motive of the multiplicative group~$\mathbb{G}_m$, shifted by~$-1$ (we work with cohomological conventions). The triangulated subcategory of~$\mathsf{DM}(\Q)$ generated by these objects is denoted by~$\mathsf{DMT}(\Q)$. By using the relation between higher Chow groups and rational~$K$-theory~\cite{blochktheory} and Borel's computation of the rational~$K$-theory of number fields~\cite{borel}, Levine defined a natural~$t$-structure on~$\mathsf{DMT}(\Q)$~\cite{levinetatemotives}. The heart of this~$t$-structure is denoted by~$\mathsf{MT}(\Q)$ and called the category of \emph{mixed Tate motives} over~$\Q$. It is a (neutral) tannakian~$\Q$-linear category which contains the objects~$\Q(-n)$.
		
		There is a faithful and exact functor 
		\begin{equation}\label{eq: realization functor Q}
		\mathsf{MT}(\mathbb{Q}) \rightarrow \mathsf{MHTS}
		\end{equation}
		 from~$\mathsf{MT}(\mathbb{Q})$ to the category~$\mathsf{MHTS}$ of mixed Hodge--Tate structures, which is called the \emph{Hodge realization functor} (\cite[\S 2.13]{delignegoncharov}, see also~\cite{huberrealization,huberrealizationcorrigendum}). It sends the object~$\Q(-n)\in\mathsf{MT}(\Q)$ to the object~$\Q(-n)\in\mathsf{MHTS}$. Composing (\ref{eq: realization functor Q}) with the fiber functors (\ref{eq: realization MHTS}) gives the Betti and de Rham realization functors, still denoted by
		\begin{equation}\label{eq: realization MTQ}
		\omega_\B:\mathsf{MT}(\Q)\rightarrow \mathsf{Vect}_\Q \hspace{1cm} \textnormal{and} \hspace{1cm} \omega_\dR:\mathsf{MT}(\Q)\rightarrow \mathsf{Vect}_\Q\ ,
		\end{equation}
		and we still have a comparison isomorphism (\ref{eq: comparison MHTS}). We note that any object in~$\mathsf{MT}(\mathbb{Q})$ is endowed with a canonical weight filtration~$W$ by subobjects such that the morphisms in~$\mathsf{MT}(\mathbb{Q})$ are strictly compatible with~$W$. The realization morphisms are compatible with the weight filtrations.
		
		\begin{rem}\label{rem: conservativity}
		The functors (\ref{eq: realization MTQ}) are fiber functors for the tannakian category~$\mathsf{MT}(\Q)$. In particular, they are conservative.
		\end{rem}
		
		 The extension groups between the objects~$\Q(-n)$ are computed by the rational~$K$-theory of~$\Q$~\cite[\S 4]{levinetatemotives} and hence given, after Borel~\cite{borel}, by
		\begin{equation}\label{eq: extensions MTQ}
		\mathrm{Ext}^1_{\mathsf{MT}(\mathbb{Q})}(\Q(-n),\Q(0)) = \begin{cases} \bigoplus_{p \textnormal{ prime}}\Q & \textnormal{ if } n=1; \\ \Q & \textnormal{ if } n \textnormal{ is odd} \geq 3 ; \\ 0 & \textnormal{ otherwise.}\end{cases}
		\end{equation}
		
		As in the category~$\mathsf{MHTS}$, the higher extension groups vanish in~$\mathsf{MT}(\Q)$. The morphisms
		\begin{equation}\label{eq: realization ext}
		\mathrm{Ext}^1_{\mathsf{MT}(\mathbb{Q})}(\Q(-n),\Q(0))\longrightarrow \mathrm{Ext}^1_{\mathsf{MHTS}}(\Q(-n),\Q(0))\cong\C/(2\pi i)^n\Q
		\end{equation}
		induced by (\ref{eq: realization functor Q}) are easy to describe. For~$n=1$, the image of the direct summand indexed by a prime~$p$ is the line spanned by the class of the Kummer extension of parameter~$p$, i.e. by~$\log(p)\in\mathbb{C}/(2\pi i)\mathbb{Q}$. For~$n\geq 3$ odd, the image is the line spanned by~$\zeta(n)\in\mathbb{C}/(2\pi i)^n\mathbb{Q}$. Thus, the morphism (\ref{eq: realization ext}) is injective for every~$n$. This implies the following theorem~\cite[Proposition 2.14]{delignegoncharov}.
		
		\begin{thm}\label{thm: fully faithful}
		The realization functor (\ref{eq: realization functor Q}) is fully faithful.
		\end{thm}
		
		This theorem is very helpful, since it allows one to compute in the category~$\mathsf{MT}(\Q)$ with period matrices; in other words, a mixed Tate motive over~$\Q$ is uniquely determined by its period matrix.
	
	\subsection{Mixed Tate motives over~$\mathbb{Z}$}
	
		Let~$\mathsf{MT}(\mathbb{Z})$ denote the category of mixed Tate motives over~$\mathbb{Z}$, as defined in~\cite{delignegoncharov}. By definition, it is a full tannakian subcategory 
		$$\mathsf{MT}(\mathbb{Z})\hookrightarrow \mathsf{MT}(\mathbb{Q})$$
		of the category of mixed Tate motives over~$\mathbb{Q}$, which contains the pure Tate motives~$\Q(-n)$ for every integer~$n$. An object of $\mathsf{MT}(\mathbb{Q})$ is in $\mathsf{MT}(\mathbb{Z})$ if and only if it has no subquotient isomorphic to a non-split extension of $\Q(-n)$ by $\Q(-n+1)$.
		
		 The extension groups in the category $\mathsf{MT}(\mathbb{Z})$ satisfy the following properties:
		\begin{enumerate}[1.]
		\item~$\mathrm{Ext}^1_{\mathsf{MT}(\mathbb{Z})}(\Q(-1),\Q(0))=0$;
		\item the natural morphism~$\mathrm{Ext}^1_{\mathsf{MT}(\mathbb{Z})}(\Q(-n),\Q(0))\rightarrow \mathrm{Ext}^1_{\mathsf{MT}(\Q)}(\Q(-n),\Q(0))$ is an isomorphism for~$n\neq 1$.
		\end{enumerate}
		 As in the categories~$\mathsf{MHTS}$ and~$\mathsf{MT}(\Q)$, the higher extension groups vanish in~$\mathsf{MT}(\mathbb{Z})$.\\
				 	
		For~$n\geq 3$ odd, there is an essentially unique non-trivial extension of~$\Q(-n)$ by~$\Q(0)$ in the category~$\mathsf{MT}(\Q)$, which actually lives in~$\mathsf{MT}(\mathbb{Z})$. A period matrix for such an extension is 
		$$ \left( \begin{array}{cc}
		1 & \zeta(n) \\
		0 & (2\pi i)^{n}
		\end{array} \right).$$
		Apart from the case~$n=3$ (see~\cite[Corollary 11.3]{browndinnerparties} or Proposition~\ref{prop: period matrix Z odd} below), we do not know of any \emph{geometric} construction of these extensions.
		
\section{Definition of the zeta motives~$\Z^{(n)}$}\label{section3}

	We define the zeta motives~$\Z^{(n)}$ and explain how to define elements of their Betti and de Rham realizations. In particular, we define the classes of the Eulerian differential forms, which are elements of the de Rham realization~$\Z^{(n)}_\dR$ constructed out of the family of Eulerian polynomials. We also note that the zeta motives fit into an inductive system~$\cdots \rightarrow \Z^{(n-1)}\rightarrow \Z^{(n)}\rightarrow \cdots$ which is compatible with the Eulerian differential forms.

	\subsection{The definition}\label{par: def Z}
	
		Let~$n\geq 1$ be an integer. In the affine~$n$-space~$X_n=\mathbb{A}^n_\Q$ we consider the hypersurfaces
		$$A_n=\{x_1\cdots x_n=1\} \; \; \textnormal{ and }$$
		$$B_n=\bigcup_{1\leq i\leq n } \{x_i=0\}\cup\bigcup_{1\leq i\leq n}\{x_i=1\}\ .$$
		The union~$A_n\cup B_n$ is almost a normal crossing divisor inside~$X_n$: around the point~$P_n=(1,\ldots,1)$, it looks like~$z_1\cdots z_n(z_1+\cdots+z_n)=0$ (set~$x_i=\exp(z_i)$). Let 
		$$\pi_n:\td{X}_n\rightarrow X_n$$
		be the blow-up at~$P_n$, and~$E_n=\pi_n^{-1}(P_n)$ be the exceptional divisor. We denote respectively by~$\td{A}_n$ and~$\td{B}_n$ the strict transforms of~$A_n$ and~$B_n$ along~$\pi_n$. The union~$\td{A}_n\cup\td{B}_n\cup E_n$ is a simple normal crossing divisor inside~$\td{X}_n$. 
		
		There is an object~$\Z^{(n)}\in\mathsf{MT}(\Q)$, which we may abusively denote by
		$$\Z^{(n)}=H^n(\td{X}_n - \td{A}_n,(\td{B}_n\cup E_n) - (\td{B}_n\cup E_n)\cap\td{A}_n)\ ,$$
		such that its Betti and de Rham realizations (\ref{eq: realization MTQ}) are ($?\in\{\B,\dR\}$)
		$$\Z^{(n)}_{\,?}=H^n_?(\td{X}_n - \td{A}_n,(\td{B}_n\cup E_n) - (\td{B}_n\cup E_n)\cap\td{A}_n) \ .$$
		
		 We now give the precise definition of~$\Z^{(n)}$, along the lines of~\cite[Proposition 3.6]{goncharovperiodsmm}. Let us write~$Y=\td{X}_n-\td{A}_n$ and~$\partial Y=(\td{B}_n\cup E_n) - (\td{B}_n\cup E_n)\cap\td{A}_n$, viewed as schemes defined over~$\Q$. We have a decomposition into smooth irreducible components~$\partial Y=\bigcup_i\partial_i Y$, where~$i$ runs in a set of cardinality~$2n+1$. For a set~$I=\{i_1,\ldots,i_r\}$ of indices, we denote by~$\partial_I Y=\partial_{i_1}Y\cap\cdots\cap\partial_{i_r}Y$ the corresponding intersection; it is either empty or a smooth subvariety of~$X$ of codimension~$r$.
		
		We thus get an object
		\begin{equation}\label{eq: complex of varieties}
		\cdots \rightarrow \bigsqcup_{|I|=3}\partial_I Y \rightarrow \bigsqcup_{|I|=2}\partial_I Y \rightarrow \bigsqcup_{|I|=1}\partial_I Y \rightarrow Y \rightarrow 0
		\end{equation}
		in Voevodsky's triangulated category~$\mathsf{DM}(\Q)$, see \S\ref{par: def MT}. The differentials are the alternating sums of the natural closed immersions. One readily checks that the complex (\ref{eq: complex of varieties}) lives in the triangulated subcategory~$\mathsf{DMT}(\Q)$. By definition, the object~$\Z^{(n)}$ in~$\mathsf{MT}(\Q)$ is the~$n$-th cohomology group of the complex (\ref{eq: complex of varieties}) with respect to the~$t$-structure.
		
		\begin{defi}
		For~$n\geq 1$, we call~$\Z^{(n)}\in\mathsf{MT}(\Q)$ the \emph{$n$-th zeta motive}.
		\end{defi}
		
		\begin{rem}\label{rem: Z1Q0}
		For~$n=1$, the blow-up map~$\pi_1:\td{X}_1\rightarrow X_1$ is an isomorphism and~$\td{A}_1=\varnothing$, so that we get~$\Z^{(1)}=H^1(\mathbb{A}^1_\Q,\{0,1\})$. We have a long exact sequence in relative cohomology
		$$0\rightarrow H^0(\mathbb{A}^1_\Q,\{0,1\})\rightarrow H^0(\mathbb{A}^1_\Q) \rightarrow H^0(\{0\})\oplus H^0(\{1\}) \rightarrow \Z^{(1)} \rightarrow 0\ ,$$
		which shows that~$H^0(\mathbb{A}^1_\Q,\{0,1\})=0$ and that we have an isomorphism~$\Z^{(1)}\simeq \Q(0)$.
		\end{rem}
		
		\begin{rem}
		We will prove in Proposition~\ref{prop: Z MTZ} that~$\Z^{(n)}$ is actually an object of the full subcategory~$\mathsf{MT}(\mathbb{Z})\hookrightarrow \mathsf{MT}(\Q)$. It would be possible, but a little technical, to prove it directly from the definition by using the criterion~\cite[Proposition 4.3]{goncharovmanin} on some compactification of~$\td{X}_n-\td{A}_n$.
		\end{rem}
		
	\subsection{Betti and de Rham realizations, 1}
	
		We now give a first description of the Betti and de Rham realizations of the zeta motive~$\Z^{(n)}$.
		
		We let~$C_\bullet$ denote the functor which assigns to a topological space the complex of singular chains with rational coefficients. By definition, the dual of the Betti realization~$\Z^{(n),\vee}_\B$ is the~$n$-th homology group of the total complex of the double complex
		\begin{equation}\label{eq: double complex Betti}
		\begin{gathered}\xymatrix{
		 \ar@{.>}[r] & \displaystyle\bigoplus_{|I|=2} C_0(\partial_I Y(\C)) \ar[r]& \displaystyle\bigoplus_{|I|=1}C_0(\partial_I Y(\C)) \ar[r]& C_0(Y(\C))\\
		 & \ar@{.>}[r]\ar@{.>}[u]& \displaystyle\bigoplus_{|I|=1} C_1(\partial_I Y(\C)) \ar[r]\ar[u]& C_1(Y(\C)) \ar[u]\\
		&& \ar@{.>}[r]\ar@{.>}[u]& C_2(Y(\C)) \ar[u] \\
		&&&  \ar@{.>}[u]
		}\end{gathered}
		\end{equation}
		obtained by applying the functor~$C_\bullet$ to the complex (\ref{eq: complex of varieties}). One readily verifies that this complex is quasi-isomorphic to the quotient complex~$C_\bullet(Y(\C))/C_\bullet(\partial Y(\C))$, classically used to define the relative homology groups~$H_\bullet^\B(Y,\partial Y)=H_\bullet^{\mathrm{sing}}(Y(\C),\partial Y(\C))$.
		
		We let~$\Omega^\bullet_{\partial_I Y}$ denote the complex of sheaves of algebraic differential forms on the smooth variety~$\partial_I Y$, extended by zero to~$Y$. By definition, the de Rham realization~$\Z^{(n)}_\dR$ is the hypercohomology of the total complex of the double complex of sheaves
		
		\begin{equation}\label{eq: double complex de Rham}
		\begin{gathered}\xymatrix{
		 & \displaystyle\bigoplus_{|I|=2} \Omega^0_{\partial_I Y} \ar@{.>}[l]\ar@{.>}[d]& \displaystyle\bigoplus_{|I|=1} \Omega^0_{\partial_I Y} \ar[l]\ar[d]& \Omega^0_Y \ar[l]\ar[d]\\
		 &  & \displaystyle\bigoplus_{|I|=1} \Omega^1_{\partial_I Y} \ar@{.>}[l]\ar@{.>}[d]& \Omega^1_Y \ar[l]\ar[d]\\
		&& & \Omega^2_Y \ar@{.>}[l]\ar@{.>}[d] \\
		&&&  
		}\end{gathered}
		\end{equation}
		where the vertical arrows are the exterior derivatives and the horizontal arrows are the alternating sums of the natural restriction maps as in the complex (\ref{eq: complex of varieties}).
		
		The comparison morphism between the Betti and de Rham realizations of~$\Z^{(n)}$ is induced, after complexification, by the morphism from the double complex (\ref{eq: double complex de Rham}) to the double complex (\ref{eq: double complex Betti}) given by integration. Note that one first has to replace (\ref{eq: double complex Betti}) by the double complex of sheaves of singular cochains.
	
	\subsection{Betti and de Rham realizations, 2}

		We now give descriptions of the Betti and de Rham realizations of~$\Z^{(n)}$ that allow one to work directly in the affine space~$X_n$ and do not require to work in the blow-up~$\td{X}_n$. The justification of the blow-up process goes as follows. Suppose that one wants to find a motive whose periods include all absolutely convergent integrals of the form
		\begin{equation}\label{eq: generic integral}
		\int_{[0,1]^n}\frac{P(x_1,\ldots,x_n)}{(1-x_1\cdots x_n)^N}\, dx_1\cdots dx_n
		\end{equation}
		where~$P(x_1,\ldots,x_n)$ is a polynomial with rational coefficients, and~$N\geq 0$ is an integer. On the Betti side, we note that the boundary of~$[0,1]^n$ intersects the divisor~$A_n(\C)$ of poles of the differential forms at the point~$P_n(\C)$. The blow-up process is thus required in order to have a class that represents the integration domain. On the de Rham side, the blow-up process is required in order to only consider \emph{absolutely convergent} integrals of the form (\ref{eq: generic integral}). This is made precise by Propositions~\ref{prop: excision} and~\ref{prop: poles} below.\\
		
		We start with the Betti realization. Let us write~$\accentset{\circ}{A}_n=A_n-P_n$ and note that this is not a closed subset, but only a locally closed subset, of~$X_n$.
	
		\begin{prop}\label{prop: excision}
		The blow-up morphism~$\pi_n:\td{X}_n\rightarrow X_n$ induces an isomorphism
		$$\Z^{(n),\vee}_\B \stackrel{\cong}{\longrightarrow} H_n^{\mathrm{sing}}(X_n(\C)-\accentset{\circ}{A}_n(\C),B_n(\C)-B_n(\C)\cap\accentset{\circ}{A}_n(\C))\ .$$
		\end{prop}
		
		\begin{proof}
		The blow-up morphism~$\pi_n$ is the contraction of the exceptional divisor~$E_n$ onto the point~$P_n$. Thus, this is a consequence of the classical excision theorem in singular homology, see for instance~\cite[Proposition 2.22]{hatcherbook}.
		\end{proof}
	
		As a consequence of Proposition~\ref{prop: excision}, we see that the unit~$n$-square~$\square^n=[0,1]^n\subset X_n(\C)-\accentset{\circ}{A}_n(\C)$ defines a class 
		$$[\square^n]\in \Z^{(n),\vee}_\B\ .$$ 
		When viewed in~$\td{X}_n(\C)-\td{A}_n(\C)$, it is the class of the strict transform~$\td{\square}^n$, which has the combinatorial structure of an~$n$-cube truncated at one of its vertices.\\
	
		We now turn to a description of the de Rham realization of~$\Z^{(n)}$. Instead of giving a general description in terms of algebraic differential forms on~$X_n-A_n$, we will only give a way of defining many classes in~$\Z^{(n)}_\dR$, which will turn out to be enough for our purposes.
	
		\begin{defi}\label{def: integrable}
		An algebraic differential~$n$-form on~$X_n-A_n$ is said to be \emph{integrable} if it can be written as a linear combination of forms of the type	
		\begin{equation}\label{eq: convergent form}
		\omega=\dfrac{(1-x_1)^{v_1-1}\cdots (1-x_n)^{v_n-1}f(x_1,\ldots,x_n)}{(1-x_1\cdots x_n)^N}\, dx_1\cdots dx_n
		\end{equation}
		with~$v_1,\ldots,v_n\geq 1$ and~$N\geq 0$ integers such that~$v_1+\cdots+v_n\geq N+1$, and~$f(x_1,\ldots,x_n)$ a polynomial with rational coefficients.
		\end{defi}
	
		The terminology is justified by the following proposition.
	
		\begin{prop}\label{prop: poles}
		Let~$\omega$ be an algebraic differential~$n$-form on~$X_n-A_n$. If~$\omega$ is integrable, then~$\pi_n^*(\omega)$ does not have a pole along~$E_n$, and thus defines a class in~$\Z^{(n)}_\dR$. In particular, the integral
		$$\int_{\td{\square}^n}\pi_n^*(\omega)=\int_{\square^n}\omega$$
		is absolutely convergent and is a period of~$\Z^{(n)}$.
		\end{prop}
		
		\begin{proof}
		We write~$\omega$ as in (\ref{eq: convergent form}). We note that the only problem for absolute convergence is around the point~$(1,\ldots,1)$. Let us thus make the change of variables~$y_i=1-x_i$ for~$i=1,\ldots,n$, and~$g(y_1,\ldots,y_n)=(-1)^n\,f(x_1,\ldots,x_n)$. We write~$h(y_1,\ldots,y_n)=1-(1-y_1)\cdots (1-y_n)$ so that we have
		$$\omega=\dfrac{y_1^{v_1-1}\cdots y_n^{v_n-1} g(y_1,\ldots,y_n)}{h(y_1,\ldots,y_n)^N}\,dy_1\cdots dy_n\ .$$
		There are~$n$ natural affine charts for the blow-up~$\pi_n:\td{X}_n\rightarrow X_n$ of the point~$(0,\ldots,0)$, and by symmetry it is enough to work in the first one. We then have local coordinates~$(z_1,\ldots,z_n)$ on~$\td{X}_n$, which are linked to the coordinates~$(y_1,\ldots,y_n)=\pi_n(z_1,\ldots,z_n)$ by the formula
		$$(y_1,\ldots,y_n)=(z_1,z_1z_2,\ldots,z_1z_n)\ .$$
		The problem of convergence occurs in the neighborhood of the exceptional divisor~$E_n$, which is defined by the equation~$z_1=0$. Since~$h(0,\ldots,0)=0$, we may write
		$$h(z_1,z_1z_2,\ldots,z_1z_n)=z_1\,\td{h}(z_1,\ldots,z_n)$$ 
		with~$\td{h}(z_1,\ldots,z_n)$ a polynomial such that~$\td{h}(0,\ldots,0)=1$. The strict transform~$\td{A}_n$ of~$A_n$ is thus defined by the equation~$\td{h}(z_1,\ldots,z_n)=0$.
		We note that we have~$dy_1\cdots dy_n=z_1^{n-1}dz_1 \cdots dz_n$, so that we can write 
		$$\pi_n^*(\omega)=\dfrac{z_1^{v_1-1}(z_1z_2)^{v_2-1}\cdots (z_1z_n)^{v_n-1} g(z_1,z_1z_2,\ldots,z_1z_n)}{z_1^N\td{h}(z_1,\ldots,z_n)^N}\,z_1^{n-1} dz_1\cdots dz_n=z_1^{v_1+\cdots+v_n-N-1}\Omega\ ,$$
		where~$\Omega$ has a pole along~$\td{A}_n$ but not along~$E_n$. The claim follows.
		\end{proof}
		
		We make an abuse of notation and denote by 
		$$[\omega]\in\Z^{(n)}_\dR$$
		the class of the pullback~$\pi_n^*(\omega)$ for~$\omega$ integrable, so that the comparison isomorphism reads
		$$\langle[\square^n],[\omega]\rangle=\int_{\square^n}\omega\ .$$
		
		We note the converse of Proposition~\ref{prop: poles}, which we will not use.
		
		\begin{prop}\label{prop: polesconverse}
		Let~$\omega$ be an algebraic differential~$n$-form on~$X_n-A_n$. If the integral~$\int_{\square^n}\omega$ is absolutely convergent, then~$\omega$ is integrable.
		\end{prop}
		
		\begin{proof}
		In the coordinates~$y_i=1-x_i$, we write 
		$$\omega=\dfrac{P(y_1,\ldots,y_n)}{h(y_1,\ldots,y_n)^N}\, dy_1\cdots dy_n$$
		with~$P(y_1,\ldots,y_n)$ a polynomial with rational coefficients. If the integral~$\int_{\square^n}\omega$ is absolutely convergent in the neighborhood of the point~$(0,\cdots,0)$, then after the change of variables
		$$\phi(z_1,\ldots,z_n)=(z_1,z_1z_2,\ldots,z_1z_n)$$
		we get an absolutely convergent integral in the neighborhood of~$z_1=0$. We write, as in the proof of Proposition~\ref{prop: poles},
		$$\phi^*(\omega)=\dfrac{P(z_1,z_1z_2,\ldots,z_1z_n)}{z_1^{N-n+1}\td{h}(z_1,\ldots,z_n)^N}\, dz_1\cdots dz_n\ .$$
		Let us write 
		$$P(y_1,\ldots,y_n)=\sum_{\underline{a}}\lambda_{\underline{a}}\,y_1^{a_1-1}\cdots y_n^{a_n-1}$$
		with~$\lambda_{\underline{a}}\in\Q$ for every multi-index~$\underline{a}=(a_1,\ldots,a_n)$. We then have
		$$P(z_1,z_1z_2,\ldots,z_1z_n)=\sum_{\underline{a}}\lambda_{\underline{a}}\,z_1^{a_1+\cdots+a_n-n}z_2^{a_2-1}\cdots z_n^{a_n-1}\ .$$		
		Let~$v$ denote the smallest integer such that there exists a multi-index~$\underline{a}$ with~$|\underline{a}|:=a_1+\cdots+a_n=v$. We then have an equivalence 
		$$P(z_1,z_1z_2,\ldots,z_1z_n) \sim_{z_1\rightarrow 0} z_1^{v-n}Q(z_2,\ldots,z_n)$$
		where~$Q(z_2,\ldots,z_n)=\sum_{|\underline{a}|=v}\lambda_{\underline{a}}\,z_2^{a_2-1}\cdots z_n^{a_n-1}$. We also have the equivalence
		$$\widetilde{h}(z_1,\ldots,z_n)\sim_{z_1\rightarrow 0} 1+z_2+\cdots +z_n\ ,~$$
		from which we deduce
		$$\phi^*(\omega) \sim_{z_1\rightarrow 0} z_1^{v-N-1}dz_1 \,\,\dfrac{Q(z_2,\ldots,z_n)}{(1+z_2+\cdots+z_n)^N}\, dz_2\cdots dz_n\ .$$
		This gives an absolutely convergent integral in the neighborhood of~$z_1=0$ if and only if~$v\geq N+1$, which is exactly the integrability condition.
		\end{proof}
		
	\subsection{The Eulerian differential forms}
	
		Recall that the family of \textit{Eulerian polynomials}~$E_r(x)$,~$r\geq 0$, is defined by the equation
		\begin{equation}\label{eqdefE}
		\frac{E_r(x)}{(1-x)^{r+1}}=\sum_{j\geq 0}(j+1)^rx^j\ .
		\end{equation}
		We refer to~\cite{foataeulerian} for a survey on Eulerian polynomials. If~$r\geq 1$, then (\ref{eqdefE}) is equivalent to 
		$$\frac{E_r(x)}{(1-x)^{r+1}}=\frac{1}{x}\left(x\frac{d}{dx}\right)^r\frac{1}{1-x}\ \cdot$$
		For instance, we have~$E_0(x)=E_1(x)=1$,~$E_2(x)=1+x$,~$E_3(x)=1+4x+x^2$.
		The Eulerian polynomials satisfy the recurrence relation
		\begin{equation}\label{eqrecurrenceE}
		E_{r+1}(x)=x(1-x)E'_{r}(x)+(1+rx)E_r(x)\ .
		\end{equation}
		For integers~$n\geq 2$ and~$k=2,\ldots,n$, we define a differential form
		$$\omega_k^{(n)}=\dfrac{E_{n-k}(x_1\cdots x_n)}{(1-x_1\cdots x_n)^{n-k+1}}\, dx_1\cdots  dx_n.$$
		Note that we have~$\omega_n^{(n)}=\dfrac{dx_1\cdots dx_n}{1-x_1\cdots x_n}$. 
		
		\begin{lem}\label{lem: eulerian forms}
		For~$k=2,\ldots,n$, the form~$\omega_k^{(n)}$ defines a class~$[\omega_k^{(n)}]\in \Z^{(n)}_\dR$ and we have	
		\begin{equation}\label{eqzetad}
		\langle[\square^n],[\omega_k^{(n)}]\rangle=\int_{\square^n}\omega_k^{(n)}=\zeta(k)\ .
		\end{equation}
		\end{lem}
		
		\begin{proof}
		The first statement follows from Proposition~\ref{prop: poles}. The computation of the period is then straightforward using the definition (\ref{eqdefE}) of the Eulerian polynomials:
		$$\int_{\square^n}\omega_k^{(n)}=\sum_{j\geq 0}(j+1)^{n-k}\int_{[0,1]^n}(x_1\cdots x_n)^j\,dx_1\cdots dx_n=\sum_{j\geq 0}(j+1)^{-k}=\zeta(k)\ .$$
		\end{proof}	
		
		For every~$n\geq 0$, we define~$\omega_0^{(n)}=dx_1\cdots dx_n$; we also have the class~$[\omega^{(n)}_0]\in \Z_{n,\mathrm{dR}}$, whose pairing with the class~$[\square^n]$ is 
		$$\langle[\square^n],[\omega_0^{(n)}]\rangle=\int_{\square^n}\omega_0^{(n)}=1\ .$$
		
		We call the differential forms~$\omega_k^{(n)}$, for~$k=0,2,\ldots,n$, the \emph{Eulerian differential forms}.
	
	\subsection{An inductive system}\label{par: inductive}
	
		For~$n\geq 2$ there are natural morphisms
		\begin{equation}\label{eq: inductive}
		i^{(n)}:\Z^{(n-1)}\rightarrow \Z^{(n)}
		\end{equation}
		in the category~$\mathsf{MT}(\Q)$, that we now define. We fix the identification~$X_{n-1}=\{x_n=1\}\subset X_n$, which implies the equality~$A_{n-1}=A_n\cap X_{n-1}$. Let us set 
		$$B'_n=\bigcup_{1\leq i\leq n}\{x_i=0\}\cup\bigcup_{1\leq i\leq n-1}\{x_i=1\}\ ,$$
		so that we have~$B_n=B'_n\cup X_{n-1}$, and~$B_{n-1}=B'_n\cap X_{n-1}$.\\
		
		In the blow-up~$\td{X}_n$, we thus get an embedding~$\td{X}_{n-1}\subset \td{X}_n$ and identifications~$\td{A}_{n-1}=\td{A}_n\cap \td{X}_{n-1}$,~$\td{B}_{n-1}=\td{B}'_n\cap \td{X}_{n-1}$ and~$E_{n-1}=E_n\cap\td{X}_{n-1}$.
		Thus, the complex in~$\mathsf{DM}(\Q)$ that we have used to define~$\Z^{(n-1)}$ is the subcomplex
		\begin{equation}\label{eq: subcomplex of varieties}
		\cdots \rightarrow \bigsqcup_{\substack{|I|=3\\ \partial_I Y\subset \td{X}_{n-1}}}\partial_I Y \rightarrow \bigsqcup_{\substack{|I|=2\\ \partial_I Y\subset \td{X}_{n-1}}}\partial_I Y \rightarrow \td{X}_{n-1} \rightarrow 0 \rightarrow 0
		\end{equation}
		of the complex (\ref{eq: complex of varieties}) that we have used to define~$\Z^{(n)}$, shifted by~$1$. Taking the~$n$-th cohomology groups with respect to the~$t$-structure gives the morphism (\ref{eq: inductive}).
		
		In the Betti and the de Rham realizations, the morphism (\ref{eq: inductive}) is also induced by the inclusion of double subcomplexes of (\ref{eq: double complex Betti}) and (\ref{eq: double complex de Rham}).\\
		
		We define the ind-motive
		$$\Z=\lim_{\stackrel{\longrightarrow}{n}}\Z^{(n)}\ ,$$
		viewed as an ind-object in the category~$\mathsf{MT}(\Q)$, and simply call it the \emph{zeta motive}.\\
		
		The map~$i^{(n),\vee}_\B:\Z^{(n),\vee}_\B\rightarrow \Z^{(n-1),\vee}_\B$ given by the transpose of the Betti realization of~$i^{(n)}$ satisfies
		\begin{equation}\label{eq: inductive square}
		i^{(n),\vee}_\B([\square^{n-1}])=[\square^n]\ .
		\end{equation}
		More generally and loosely speaking, if~$\sigma$ is a chain on~$\td{X}_n(\C)-\td{A}_n(\C)$ whose boundary is on~$\td{B}_n(\C)\cup E_n(\C)$, then~$i^{(n),\vee}_\B([\sigma])$ is the class of \enquote{the component of the boundary of~$\sigma$ that lives on~$\td{X}_{n-1}(\C)$}. According to Proposition~\ref{prop: excision}, one can also work with chains on~$X_n(\C)-\accentset{\circ}{A}_n(\C)$. We note that (\ref{eq: inductive square}) allows us to define a class
		$$[\square] \in \Z_\B^\vee:=	\lim_{\stackrel{\longleftarrow}{n}}\Z^{(n),\vee}_\B\ .$$
		
		 \begin{rem}\label{rem: signs}
		There are (alternating) signs in the differentials of the complexes (\ref{eq: complex of varieties}), (\ref{eq: double complex Betti}), (\ref{eq: double complex de Rham}), that we leave to the reader. This also induces signs on the different components of the inclusions of subcomplexes such as (\ref{eq: subcomplex of varieties}); these signs are fixed once and for all by equation (\ref{eq: inductive square}).
		\end{rem}
		
		The next proposition shows that the Eulerian differential forms~$\omega_{k}^{(n)}$ are compatible with the inductive structure on the zeta motives.
		
		\begin{prop}\label{propdRrecurrence}
		For integers~$n\geq 2$ and~$k=0,2,\ldots,n-1$, the map~$i^{(n)}_\dR:\Z^{(n-1)}_\dR\rightarrow \Z^{(n)}_\dR$ sends the class~$[\omega_{k}^{(n-1)}]$ to the class~$[\omega_k^{(n)}]$.
		\end{prop}
		
		\begin{proof}
		Since all the differential forms that we are manipulating have no poles along the exceptional divisors~$E_{n-1}$ and~$E_n$, it is safe to do the computations in the affine spaces~$X_{n-1}$ and~$X_n$; we leave it to the reader to turn them into computations in~$\td{X}_{n-1}$ and~$\td{X}_n$ by working in local charts as in the proof of Proposition~\ref{prop: poles}. Let us first assume that~$k\in\{2,\ldots,n-1\}$. We put
		$$\eta^{(n-1)}_{k}=\dfrac{x_{n}E_{n-1-k}(x_1\cdots x_{n})}{(1-x_1\cdots x_{n})^{n-k}}\, dx_1\cdots dx_{n-1}\ ,$$
		viewed as a form on~$X_{n}$. Then we have~$(\eta^{(n-1)}_{k})_{|X_{n-1}}=\omega^{(n-1)}_{k}$ and~$(\eta^{(n-1)}_{k})_{|B'_{n-1}}=0$. A diagram chase in the double complex (\ref{eq: double complex de Rham}) shows that~$i^{(n)}_\dR([\omega^{(n-1)}_{k}])$ is the class of~$$(-1)^{n-1}\,(d(\eta^{(n-1)}_{k}))$$ 
		(the sign is here to be consistent with the Betti version, see Remark~\ref{rem: signs}). We have
		$$(-1)^{n-1}\, d(\eta^{(n-1)}_{k})=\frac{\partial}{\partial x_{n}}\left(\dfrac{x_{n}E_{n-1-k}(x_1\cdots x_{n})}{(1-x_1\cdots x_{n})^{n-k}}\right) dx_1\cdots dx_{n}~$$
		and one easily sees that setting~$x=x_1\cdots x_{n}$ we have
		$$\frac{\partial}{\partial x_{n}}\left(\dfrac{x_{n}E_{n-1-k}(x_1\cdots x_{n})}{(1-x_1\cdots x_{n})^{n-k}}\right)=\dfrac{x(1-x)E'_{n-1-k}(x)+(1+(n-1-k)x)E_{n-1-k}(x)}{(1-x)^{n-k+1}}\ .$$
		Using the recurrence relation (\ref{eqrecurrenceE}), one then concludes that
		$$(-1)^{n-1}d(\eta_{k}^{(n-1)})=\dfrac{E_{n-k}(x_1\cdots x_{n})}{(1-x_1\cdots x_{n})^{n-k+1}}\, dx_1\cdots dx_{n}=\omega_{k}^{(n)}.$$
		For~$k=0$, this is the same computation with~$\eta^{(n)}_0=x_{n}\, dx_1\cdots dx_{n-1}$ and 
		$$(-1)^{n-1}d(\eta^{(n-1)}_{0})=dx_1\cdots dx_{n}=\omega_{0}^{(n)}.$$
		\end{proof}
		
		Proposition~\ref{propdRrecurrence} allows us to unambiguously define classes 
		$$[\omega_k]\in \Z_{\dR}$$
		for~$k=0,2,3,\ldots$, whose pairing with the class~$[\square]\in\Z_\B^\vee$ is~$$\langle[\square],[\omega_0]\rangle=1 \;\; \textnormal{ and } \;\; \langle[\square],[\omega_k]\rangle=\zeta(k)\;\; (k\geq 2)\ .$$
		
		\begin{rem}
		The proof of Proposition~\ref{propdRrecurrence} can be thought of as a cohomological version of the relation
		$$\int_{\square^n}\omega_{k}^{(n)}=\int_{\square^{n-1}}\omega_{k}^{(n-1)},$$
		which may be proved using Stokes's theorem and the recurrence relation (\ref{eqrecurrenceE}). 
		\end{rem}
		
		\begin{prop}\label{prop: eulerian forms weights}
		For integers~$n\geq 1$ and~$k=0,2,\ldots,n$, the class~$[\omega_k^{(n)}]$ lives in the pure weight~$2k$ component of~$\Z^{(n)}_\dR$.
		\end{prop}
		
		\begin{proof}
		For~$k=0$, Proposition~\ref{propdRrecurrence} and the fact that the maps~$i^{(n)}_\dR$ are compatible with the weight gradings implies that it is enough to do the proof for~$n=1$; this case is easy since~$\Z^{(1)}\cong\Q(0)$ only has weight~$0$. We now turn to the case~$k=2,\ldots,n$. Thanks to Proposition~\ref{propdRrecurrence} and the fact that the maps~$i^{(n)}_\dR$ are compatible with the weight gradings, it is enough to check it for~$k=n$.
		
		By (\ref{eq: splitting W}), we need to prove that the class of~$\omega_n^{(n)}$ is in~$F^n\Z^{(n)}_\dR$. Let~$Y$ be a smooth projective variety of dimension~$n$,~$D$ be a normal crossing divisor inside~$Y$, and~$Z$ be a closed subvariety of~$Y$ of dimension less than or equal to $n-1$. Then we have 
		$$F^nH^n_\dR(Y-D,Z-Z\cap D) = \mathrm{Im}\left(H^0(\Omega^n_Y(\log D))\longrightarrow H^n_\dR(Y-D,Z-Z\cap D)\right)\ .$$
		Thus, it is enough to prove that there is a compactification~$Y$ of~$\widetilde{X}_n-\widetilde{A}_n$ such that~$Y-(\widetilde{X}_n-\widetilde{A}_n)$ is a normal crossing divisor~$D$, and such that~$\omega_n^{(n)}$ has logarithmic singularities along~$D$. Since~$\omega_n^{(n)}$ does not have poles along~$E_n$, we can work on~$X_n-A_n$ instead. Let us start with~$Y_1=(\mathbb{P}^1)^n$ with coordinates~$((x_1:y_1),\ldots,(x_n:y_n))$, and~$D_1$ the divisor given by the union of the subvarieties~$\{y_i=0\}$ for~$i=1,\ldots, n$, and the subvariety~$\{x_1\cdots x_n=y_1\cdots y_n\}$ (this is the closure of~$A_n$). This is not enough since $D_1$ is not a normal crossing divisor. We then finish thanks to the following lemma.
		\end{proof}
		
		\begin{lem}
		Let~$\phi:Y\rightarrow Y_1$ be the iterated blow-up of all the codimension-$2$ subvarieties~$Z_{i,j}=\{y_i=x_j=0\}$,~$i\neq j$, in any order. Then~$D=\phi^{-1}(D_1)$ is a normal crossing divisor inside~$Y$ and~$\phi^*(\omega^{(n)}_n)$ has logarithmic singularities along~$D$.
		\end{lem}
		
		\begin{proof}
		This is checked locally on the standard affine cover of~$Y$, consisting of~$2^n$ affine spaces. By symmetry of the variables, it is enough to look at the charts~$U_r$ with affine coordinates
		$$(y_1,\ldots,y_r,x_{r+1},\ldots,x_n)\ ,$$
		for~$r=0,\ldots,n$. We note that in the chart~$U_r$ the divisor~$D_1$ is the union of the subvarieties~$\{y_i=0\}$ for~$i=1,\ldots,r$, and the subvariety~$\{y_1\cdots y_r=x_{r+1}\cdots x_n\}$. In that chart the differential form that we are looking at is (up to a sign)
		$$\omega=\dfrac{dy_1\cdots dy_r \, dx_{r+1}\cdots dx_n}{y_1\cdots y_r(y_1\cdots y_r-x_{r+1}\cdots x_n)}\cdot$$
		We proceed by induction on~$r$. For~$r=0$,~$D_1$ only consists of~$\{x_1\cdots x_n=1\}$, which is a normal crossing divisor, and~$\omega$ has logarithmic singularities along~$D_1$. The subvarieties~$Z_{i,j}$ do not intersect~$U_0$, so the blow-ups do not change anything. For a given~$r=1,\ldots,n$, let us look at the blow-up of~$Z_{1,n}=\{y_1=x_n=0\}$ (this is enough for reasons of symmetry) in the chart~$U_r$. There are two natural affine charts~$\mathbb{A}^n\rightarrow U_r$ for the blow-up.
		\begin{enumerate}
		\item On the first chart, the blow-up map is given by 
		$$(y_1,\ldots,y_r,x_{r+1},\ldots,x_n)=(v_1,\ldots,v_r,u_{r+1},\ldots,u_{n-1},v_1u_n)\ .$$
		The preimage of~$D_1$ consists of the subvarieties~$\{v_i=0\}$ for~$i=1,\ldots,r$, and the subvariety~$\{v_2\cdots v_r=u_{r+1}\cdots u_n\}$. The pullback of~$\omega$ is
		$$\widetilde{\omega}=\frac{dv_1\cdots dv_r\, du_{r+1}\cdots du_n}{v_1\cdots v_r(v_2\cdots v_r-u_{r+1}\cdots u_n)}= \frac{dv_1}{v_1}\wedge\omega' \;\textnormal{ with }\; \omega'=\frac{dv_2\cdots dv_r\, du_{r+1}\cdots du_n}{v_2\cdots v_r(v_2\cdots v_r-u_{r+1}\cdots u_n)}\cdot$$
		We note that~$\{v_1=0\}$ is the exceptional divisor and that the total transforms of the subvarieties~$Z_{1,j}$ are empty in this chart. By the induction hypothesis (with $n$ replaced by $n-1$), the pullback of $\omega'$ by the successive blow-up of the subvarieties $Z_{i,j}$ with $i\neq 1$ has logarithmic singularities. Since $dv_1/v_1$ has logarithmic singularities along $\{v_1=0\}$, we are done.
		\item On the second chart, the blow-up map is given by 
		$$(y_1,\ldots,y_r,x_{r+1},\ldots,x_n)=(v_1u_n,v_2,\ldots,v_r,u_{r+1},\ldots,u_n)\ .$$
		The same argument as in the first chart applies.
		\end{enumerate}
		\end{proof}

	\subsection{A long exact sequence}
		
		We now show that the morphism~$i^{(n)}:\Z^{(n-1)}\rightarrow\Z^{(n)}$ fits into a long exact sequence. We first define objects of~$\mathsf{MT}(\Q)$:
		$$\Z^{(n),r}=H^r(\td{X}_n - \td{A}_n,(\td{B}_n\cup E_n) - (\td{B}_n\cup E_n)\cap\td{A}_n)$$
		and
		$${}'\Z^{(n),r}=H^r(\td{X}_n - \td{A}_n,(\td{B}'_n\cup E_n) - (\td{B}'_n\cup E_n)\cap\td{A}_n) \ ,$$
		so that~$\Z^{(n)}=\Z^{(n),n}$. We leave it to the reader to fill in the technical definitions of these objects by mimicking that of~$\Z^{(n)}$ from \S\ref{par: def Z}.
		
		\begin{prop}\label{prop: long exact sequence Z prime}
		For~$n\geq 2$, we have a long exact sequence in~$\mathsf{MT}(\Q)$:
		\begin{equation}\label{eq: long exact sequence Z}
		\cdots \rightarrow \Z^{(n-1),r-1} \rightarrow \Z^{(n),r} \rightarrow {}'\Z^{(n),r} \rightarrow  \Z^{(n-1),r} \rightarrow \Z^{(n),r+1} \rightarrow\cdots
		\end{equation}
		\end{prop}
		
		\begin{proof}
		The objects~$\Z^{(n-1),\bullet}$,~$\Z^{(n),\bullet}$ and~${}'\Z^{(n),\bullet}$ are defined via objects in~$\mathsf{DMT}(\Q)$ that we denote by~$C^{(n-1)}$,~$C^{(n)}$ and~${}'C^{(n)}$ respectively,~$C^{(n)}$ being the complex (\ref{eq: complex of varieties}) and~$C^{(n-1)}$ the subcomplex (\ref{eq: subcomplex of varieties}). Now there is an obvious exact triangle
		$$C^{(n-1)}[-1]\longrightarrow C^{(n)} \longrightarrow {}'C^{(n)} \stackrel{+1}{\longrightarrow}\ ,$$
		in~$\mathsf{DMT}(\Q)$, which gives the desired long exact sequence after taking the cohomology with respect to the~$t$-structure.
		\end{proof}
		
		We note that the map~$\Z^{(n-1),n-1}\rightarrow \Z^{(n),n}$ in the long exact sequence \eqref{eq: long exact sequence Z} is exactly~$i^{(n)}$.

\section{Computation of the zeta motives~$\Z^{(n)}$}\label{section4}

	This section is the technical heart of this article, where we compute (Theorem~\ref{thm: period matrix Z}) the full period matrix of the zeta motives~$\Z^{(n)}$. The main difficulty is showing that the motives~$\T^{(n)}$, introduced below, are semi-simple. For that we use the involution~$\tau$ defined in the introduction and the computation of the extension groups in the category~$\mathsf{MT}(\Q)$. We then define the odd zeta motive and compute its period matrix. We finish with an elementary (Hodge-theoretic) proof that the motives~$\T^{(n)}$ are semi-simple.

	\subsection{The Gysin long exact sequence}
	
		Since the divisor~$A_n$ is smooth, it is natural to decompose the motives~$\Z^{(n),r}$ thanks to a Gysin long exact sequence. In the next proposition, the definition of the objects~$H^\bullet(X_n,B_n)$ and~$H^\bullet(A_n,B_n\cap A_n)$ of~$\mathsf{MT}(\Q)$ is similar to that of~$\Z^{(n)}$ from \S\ref{par: def Z}.
	
		\begin{prop}\label{prop: Gysin long exact sequence}
		For~$n\geq 1$, we have a long exact sequence in~$\mathsf{MT}(\Q)$:
		\begin{equation}\label{eq: Gysin long exact sequence}
		\cdots \rightarrow H^r(X_n,B_n) \rightarrow \Z^{(n),r} \rightarrow H^{r-1}(A_n,B_n\cap A_n)(-1) \rightarrow H^{r+1}(X_n,B_n) \rightarrow \Z^{(n),r+1} \rightarrow\cdots 
		\end{equation}
		\end{prop}
		
		\begin{proof}
		Recall from~\cite[(3.5.4)]{voevodskytriangulated} the existence of a Gysin exact triangle in the category~$\mathsf{DM}(\Q)$. For the pair~$(\td{X}_n,\td{A}_n)$, it reads (with cohomological conventions)
		$$\td{X}_n\longrightarrow \td{X}_n-\td{A}_n \longrightarrow \td{A}_n(-1)[-1] \stackrel{+1}{\longrightarrow}$$
		and is an exact triangle in the category~$\mathsf{DMT}(\Q)$. Applying this triangle to every pair~$(\partial_I Y,\partial_I Y\cap\td{A}_n)$ in the complex (\ref{eq: complex of varieties}) and taking the cohomology with respect to the~$t$-structure leads to a long exact sequence
		$$\cdots\rightarrow H^r(\td{X}_n,\td{B}_n\cup E_n) \rightarrow H^r(\td{X}_n - \td{A}_n,(\td{B}_n\cup E_n) - \td{A}_n) \rightarrow H^{r-1}(\td{A}_n,(\td{B}_n\cup E_n)\cap \td{A}_n)(-1)  \rightarrow\cdots~$$
		in~$\mathsf{MT}(\Q)$. One finishes with the fact that the natural morphisms
		$$H^r(\td{X}_n,\td{B}_n\cup E_n)\rightarrow H^r(X_n,B_n) \;\;\textnormal{ and }\;\; H^{r-1}(\td{A}_n,(\td{B}_n\cup E_n)\cap \td{A}_n)\rightarrow H^{r-1}(A_n,B_n\cap A_n)$$
		are isomorphisms. This can be checked in the Betti realization (see Remark~\ref{rem: conservativity}), where it is a consequence of the excision theorem as in the proof of Proposition~\ref{prop: excision}.
		\end{proof}

	\subsection{The motives~$H^\bullet(X_n,B_n)$}
		
		The computation of the motives~$H^\bullet(X_n,B_n)$ appearing in the long exact sequence (\ref{eq: Gysin long exact sequence}) is relatively easy.	
	
		\begin{prop}\label{prop: computation HXB}
		\begin{enumerate}
		\item We have~$H^r(X_n,B_n)=0$ for~$r\neq n$, and an isomorphism~$H^n(X_n,B_n)\cong \Q(0)$.
		\item A basis for the de Rham realization~$H^n_\dR(X_n,B_n)$ is the class of the form~$dx_1\cdots dx_n$.
		\item A basis for the Betti realization~$H_n^\B(X_n,B_n)$ is the class of the unit~$n$-cube~$\square^n=[0,1]^n$.
		\end{enumerate}
		\end{prop}
		
		\begin{proof}
		By the relative K\"{u}nneth formula we have~$H^\bullet(X_n,B_n)\cong H^\bullet(X_1,B_1)^{\otimes n}$ so that it is enough to prove the proposition for~$n=1$. This has already been done in Remark~\ref{rem: Z1Q0}.
		\end{proof}
		
	
	\subsection{The motives~$H^\bullet(A_n,B_n\cap A_n)$}
	
		For~$n\geq 1$, we realize the~$n$-torus as~$T^n=\{x_1\cdots x_{n+1}=1\}$, and we have subtori~$T^{n-1}_i=\{x_i=1\}\subset T^n$ for~$i=1,\ldots,n+1$. We define
		$$\T^{(n),r}=H^r(T^n,\textstyle\bigcup_{1\leq i\leq n+1}T^{n-1}_i) \;\; \textnormal{ and } \;\; {}'\T^{(n),r}=H^r(T^n,\textstyle\bigcup_{1\leq i\leq n}T^{n-1}_i)\ ,$$
		which are objects in~$\mathsf{MT}(\Q)$ (whose definition is similar to that of~$\Z^{(n)}$ from \S\ref{par: def Z}) and write
		$$\T^{(n)}=\T^{(n),n} \;\;\textnormal{ and }\;\; {}'\T^{(n)}={}'\T^{(n),n}\ .$$ 
		We then have 
		$$H^{r-1}(A_n,B_n\cap A_n)\cong\T^{(n-1),r-1}\ .$$ 
		By mimicking the proof of Proposition~\ref{prop: long exact sequence Z prime}, one produces a long exact sequence in~$\mathsf{MT}(\Q)$:
		\begin{equation}\label{eq: long exact sequence T}
		\cdots \rightarrow \T^{(n-1),r-1} \rightarrow \T^{(n),r} \rightarrow {}'\T^{(n),r} \rightarrow  \T^{(n-1),r} \rightarrow \T^{(n),r+1}\rightarrow \cdots
		\end{equation}
		
		\begin{prop}\label{prop: computation grT}
		\begin{enumerate}
		\item We have~${}'\T^{(n),r}=0$ for~$r\neq n$, and an isomorphism~${}'\T^{(n)}\cong H^n(T^n)\cong \Q(-n)$.
		\item We have~$\T^{(n),r}=0$ for~$r\neq n$, and short exact sequences in~$\mathsf{MT}(\Q)$:
		\begin{equation}\label{eq: short exact sequence T}
		0\rightarrow \T^{(n-1)} \stackrel{j^{(n)}}{\longrightarrow} \T^{(n)} \rightarrow H^n(T^n)\rightarrow 0\ .
		\end{equation}
		\end{enumerate}
		\end{prop}	
		
		\begin{proof}
		If (1) is proved, then (2) follows from the long exact sequence (\ref{eq: long exact sequence T}). By choosing coordinates~$(x_1,\ldots,x_n)$ on~$T^n$ we see that we have
		$${}'\T^{(n),\bullet}=H^\bullet((\mathbb{A}^1_\Q-\{0\})^n, \cup_{1\leq i\leq n}\{x_i=1\})\cong H^\bullet(\mathbb{A}^1_\Q-\{0\},\{1\})^{\otimes n} = ({}'\T^{(1),\bullet})^{\otimes n}\ ,$$
		where we have used the relative K\"{u}nneth formula. Thus, it is enough to prove (1) for~$n=1$, which is easy since~${}'\T^{(1),\bullet}$ is nothing but the reduced cohomology of~$\mathbb{A}^1_\Q-\{0\}$.
		\end{proof}

		\begin{rem}\label{rem: inductive T} 
		We note that the morphism~$j^{(n)}:\T^{(n-1)}\rightarrow \T^{(n)}$ in (\ref{eq: short exact sequence T}) is defined in the same way as the morphism~$i^{(n)}:\Z^{(n-1)}\rightarrow \Z^{(n)}$ from \S\ref{par: inductive}. 
		\end{rem}
		
		We note that we have~$\T^{(0)}=H^0(\mathrm{pt},\mathrm{pt})=0$, so that Proposition~\ref{prop: computation grT} implies that we have 
		$$\mathrm{gr}_{2k}^W\T^{(n)} = \begin{cases} \Q(-k) & \textnormal{ if } k\in\{1,\ldots,n\}; \\ 0 & \textnormal{ otherwise.}\end{cases}$$

		In the next proposition, we will prove that the weight filtration of~$\T^{(n)}$ actually splits in~$\mathsf{MT}(\Q)$. For that we introduce the involution~$\tau$ which acts on the tori~$T^n$ by 
		$$\tau:(x_1,\ldots,x_{n+1})\mapsto (x_1^{-1},\ldots,x_{n+1}^{-1})\ .$$
		This induces an involution, still denoted by~$\tau$, on the objects~$\T^{(n),r}$ and~${}'\T^{(n),r}$ of~$\mathsf{MT}(\Q)$, such that all the maps in the long exact sequence (\ref{eq: long exact sequence T}) commute with~$\tau$.
		
		\begin{prop}\label{prop: computation T}
		\begin{enumerate}
		\item The short exact sequences (\ref{eq: short exact sequence T}) split in~$\mathsf{MT}(\Q)$, hence we have isomorphisms:
		$$\T^{(n)}\cong \Q(-1)\oplus\Q(-2)\oplus\cdots\oplus\Q(-n)\ .$$
		Thus, a period matrix for~$\T^{(n)}$ is the diagonal matrix~$\,\mathrm{Diag}(2\pi i, (2\pi i)^2, \ldots,(2\pi i)^n)$.
		\item The involution~$\tau$ acts on the direct summand~$\Q(-k)$ of~$\T^{(n)}$ by multiplication by~$(-1)^k$.
		\end{enumerate}
		\end{prop}
		
		\begin{proof}
		We first note that~$\tau$ acts on~$H^1(T^1)$ by multiplication by~$-1$. It is enough to prove it in the de Rham realization, where it follows from~$\tau.\,\mathrm{dlog}(x_1)=-\mathrm{dlog}(x_1)$. Thus,~$\tau$ acts on~$\mathrm{gr}^W_{2n}\T^{(n)}\cong H^n(T^n)\cong H^1(T^1)^{\otimes n}$ by multiplication by~$(-1)^n$, and we are left with proving (1). We denote by~$\T^{(n)}=\T^{(n)}_+\oplus\T^{(n)}_-$ the direct sum decomposition of~$\T^{(n)}$ into its invariant and anti-invariant parts with respect to~$\tau$. We have to prove that we have isomorphisms 
		$$\T^{(2n)}_+\cong \T^{(2n+1)}_+ \cong \Q(-2)\oplus\Q(-4)\oplus\cdots\oplus \Q(-2n) $$
		and
		$$ \T^{(2n+1)}_-\cong \T^{(2n+2)}_-\cong \Q(-1)\oplus\Q(-3)\oplus\cdots\oplus \Q(-(2n+1))\ .$$
		We only prove the statements corresponding to the invariant parts, the statements corresponding to the anti-invariant parts being proved similarly. We use induction on~$n$, the case~$n=0$ being trivial:~$\T^{(0)}_+=\T^{(1)}_+=0$. The short exact sequences (\ref{eq: short exact sequence T}) imply that we have short exact sequences
		$$0\rightarrow \T^{(2n+1)}_+ \rightarrow \T^{(2n+2)}_+ \rightarrow \Q(-(2n+2)) \rightarrow 0 \;\;\; \textnormal{ and } \;\;\; 0\rightarrow \T^{(2n+2)}_+ \rightarrow \T^{(2n+3)}_+ \rightarrow 0 \rightarrow 0 \ .$$
		Using the induction hypothesis we see that we have 
		\begin{eqnarray*}
		\mathrm{Ext}^1_{\mathsf{MT}(\Q)}(\Q(-(2n+2)),\T^{(2n+1)}_+)  & \cong & \mathrm{Ext}^1_{\mathsf{MT}(\Q)}(\Q(-(2n+2)),\Q(-2)\oplus\Q(-4)\oplus\cdots\oplus\Q(-2n))\\
		& \cong & \bigoplus_{1\leq k\leq n}\mathrm{Ext}^1_{\mathsf{MT}(\Q)}(\Q(-2k),\Q(0))  \\
		& = & 0
		\end{eqnarray*}
		where we have used (\ref{eq: extensions MTQ}). Thus, the first short exact sequence splits. The second short exact sequence then completes the induction.
		\end{proof}
		
		\begin{rem}\label{rem: basis TdR}
		From the short exact sequences (\ref{eq: short exact sequence T}) it is clear that, for every $n$, $\T^{(n)}_\dR$ has a basis~$(w_1^{(n)},\ldots,w_n^{(n)})$ which is compatible with the weight grading, such that~$w_n^{(n)}$ is the class of the form~$\mathrm{dlog}(x_1)\wedge\cdots\wedge\mathrm{dlog}(x_n)$, and such that these bases are compatible with the short exact sequences (\ref{eq: short exact sequence T}).
		\end{rem}

	\subsection{The structure of the zeta motives}
	
		We can now determine the structure of the zeta motives~$\Z^{(n)}$, for~$n\geq 1$.
	
		\begin{thm}\label{thm: short exact sequence ZT}
		\begin{enumerate}
		\item We have a short exact sequence in~$\mathsf{MT}(\Q)$:
		\begin{equation}\label{eq: short exact sequence ZT}
		0\rightarrow \Q(0)\rightarrow \Z^{(n)} \stackrel{p^{(n)}}{\longrightarrow} \T^{(n-1)}(-1)\rightarrow 0\ ,
		\end{equation}
		with~$\T^{(n-1)}(-1)\cong \Q(-2)\oplus\cdots\oplus\Q(-n)$.
		\item We have a short exact sequence in~$\mathsf{MT}(\Q)$:
		\begin{equation}\label{eq: short exact sequence Z}
		0\rightarrow \Z^{(n-1)} \stackrel{i^{(n)}}{\longrightarrow} \Z^{(n)} \rightarrow \Q(-n) \rightarrow 0\ .
		\end{equation}
		\item These short exact sequences fit into a commutative diagram
		\begin{equation}\label{eq: commutative diagram 9}
		\begin{gathered}\xymatrix{
		& 0 \ar[d] & 0\ar[d] & 0\ar[d] & \\
		0 \ar[r] & \Q(0) \ar[d]^{=}\ar[r] &\Z^{(n-1)} \ar[d]^{i^{(n)}}\ar[r]^-{p^{(n-1)}} &\T^{(n-2)}(-1)\ar[d]^{j^{(n-1)}}\ar[r] & 0 \\
		0 \ar[r] & \Q(0) \ar[d]\ar[r] &\Z^{(n)} \ar[d]\ar[r]^-{p^{(n)}} &\T^{(n-1)}(-1)\ar[d]\ar[r] & 0 \\
		0 \ar[r] & 0 \ar[d]\ar[r] & \Q(-n) \ar[d]\ar[r]^{=} & \Q(-n) \ar[d]\ar[r] & 0 \\
		& 0 & 0 & 0 &
		}\end{gathered}
		\end{equation}
		where all rows and columns are exact.
		\end{enumerate}
		\end{thm}
		
		\begin{proof}
		Assertion (1) follows from Propositions~\ref{prop: Gysin long exact sequence},~\ref{prop: computation HXB} and~\ref{prop: computation T}. The commutativity of (\ref{eq: commutative diagram 9}) follows from the compatibility of the long exact sequences (\ref{eq: long exact sequence Z}) and (\ref{eq: long exact sequence T}). A diagram chase implies that (\ref{eq: short exact sequence Z}) is exact.
 		\end{proof}
 		
 		\begin{rem}
 		The difference between the sign~$(-1)^k$ in Proposition~\ref{prop: computation T} (2) and the sign~$(-1)^{k-1}$ in Theorem~\ref{thm: sign intro} comes from the Tate twist~$(-1)$ in the short exact sequence (\ref{eq: short exact sequence ZT}).
 		\end{rem}

		\begin{thm}\label{thm: period matrix Z}
		\begin{enumerate}
		\item The classes
		$$v_k^{(n)}:=[\omega_k^{(n)}] \hspace{1cm} (k=0,2,\ldots,n)$$ 
		of the Eulerian differential forms provide a basis~$(v_0^{(n)},v_2^{(n)},\ldots,v_n^{(n)})$ of the de Rham realization~$\Z^{(n)}_\dR$ which is compatible with the weight grading.
		\item There exists a unique basis~$(\varphi_0^{(n)},\varphi_2^{(n)},\ldots,\varphi_n^{(n)})$ for the dual of the Betti realization~$\Z^{(n),\vee}_\B$ which is compatible with the weight filtration and such that the period matrix for~$\Z^{(n)}$ in the~$v$-basis and the~$\varphi$-basis is
		\begin{equation}\label{eq: period matrix Z} 
		\left( \begin{array}{ccccccc}
		1 & \zeta(2) & \zeta(3) & \cdots & \cdots & \zeta(n-1) & \zeta(n) \\
		&(2\pi i)^2 &  &  &  &  & \\
		 && (2\pi i)^3 &  &  & 0 & \\
		 &&  & \ddots &  &  & \\
		 &&  & & \ddots &  & \\
		 &0&  &  &  & (2\pi i)^{n-1} & \\
		 &&  &  &  &  & (2\pi i)^n 
		\end{array} \right)\cdot
		\end{equation}
		\end{enumerate}
		\end{thm}
		
		\begin{proof}
		\begin{enumerate}
		\item Proposition~\ref{prop: eulerian forms weights} says that~$v_k^{(n)}$ is in the pure weight~$2k$ component of~$\Z^{(n)}_\dR$. Thus, it is enough to show that it is non-zero, which is a consequence of the equalities~$\langle[\square^n],v_0^{(n)}\rangle=1\neq 0$ and~$\langle[\square^n],v_k^{(n)}\rangle=\zeta(k)\neq 0$ for~$k=2,\ldots,n$.
		\item We put~$\varphi_0^{(n)}=[\square^n]$. Let~$(\psi_1^{(n-1)},\ldots,\psi_{n-1}^{(n-1)})$ be a basis of~$\T^{(n-1),\vee}_\B$ for which the period matrix is diagonal, as in Proposition~\ref{prop: computation T}. Let~$p^{(n)}$ denote the morphism~$\Z^{(n)}\rightarrow \T^{(n-1)}(-1)$, and let us consider the transpose of its Betti realization~$p^{(n),\vee}_\B:\T^{(n-1),\vee}_\B\rightarrow \Z^{(n),\vee}_\B$. Then we can put~$\varphi_k^{(n)}=p^{(n),\vee}_\B(\psi_{k-1}^{(n-1)})$ for~$k=2,\ldots,n$. The fact that this gives a basis of~$\Z^{(n),\vee}_\B$ is a consequence of the short exact sequence (\ref{eq: short exact sequence ZT}). The fact that the period matrix is as required follows from Lemma~\ref{lem: eulerian forms} and Proposition~\ref{prop: computation T}. The uniqueness statement is obvious.
		\end{enumerate} 
		\end{proof}
		
		
		We have already noted that the classes~$v_k^{(n)}$ are compatible with the inductive system of the zeta motives. By the uniqueness statement in Theorem~\ref{thm: period matrix Z}, this is also the case for the classes~$\varphi_k^{(n)}$, and the zeta motive~$\Z$ has an infinite period matrix
		$$\left( \begin{array}{ccccccc}
		1 & \zeta(2) & \zeta(3) & \zeta(4) & \cdots & \cdots  & \cdots\\
		&(2\pi i)^2 &  &  &  & &   \\
		 && (2\pi i)^3 &  &  & 0  & \\
		 &&  & (2\pi i)^4 &  & &  \\
		 &&  & & \ddots & &  \\
		 &0&  &  &  & \ddots & \\
		 &&&&&&\ddots
		\end{array} \right)\cdot$$

	\subsection{The odd zeta motive}
		
		Let us write~$\T^{(n-1)}=\T^{(n-1)}_+\oplus\T^{(n-1)}_-$ for the direct sum decomposition into its invariant and anti-invariant parts with respect to~$\tau$, and let us write~$p^{(n)}:\Z^{(n)}\rightarrow \T^{(n-1)}(-1)$ for the surjection appearing in the short exact sequence (\ref{eq: short exact sequence ZT}).
		
		\begin{defi}
		The~$n$-th \emph{odd zeta motive}~$\Z^{(n),\mathrm{odd}}$ is the object of~$\mathsf{MT}(\Q)$ defined by 
		$$\Z^{(n),\mathrm{odd}}:=(p^{(n)})^{-1}(\T^{(n-1)}_+(-1))\ .$$
		\end{defi}
		
		We obviously have a short exact sequence
		\begin{equation}\label{eq: short exact sequence Z odd}
		0\rightarrow \Q(0)\rightarrow \Z^{(n),\mathrm{odd}} \rightarrow \T^{(n-1)}_+(-1)\rightarrow 0
		\end{equation}
		with~$$\T^{(n-1)}_+(-1)\cong\bigoplus_{3\leq 2k+1\leq n}\Q(-(2k+1))\ .$$
		
		We note that there are morphisms
		$$i^{(n),\mathrm{odd}}:\Z^{(n-1),\mathrm{odd}}\rightarrow \Z^{(n),\mathrm{odd}}$$
		such that~$i^{(2n),\mathrm{odd}}$ is an isomorphism for every integer~$n$. The limit
		$$\Z^{\mathrm{odd}}:=\lim_{\stackrel{\longrightarrow}{n}}\Z^{(n),\mathrm{odd}}$$
		is an ind-object in~$\mathsf{MT}(\Q)$ that we simply call the \emph{odd zeta motive}.
		
		\begin{prop}\label{prop: period matrix Z odd}
		\begin{enumerate}
		\item We have a direct sum decomposition
		\begin{equation}\label{eq: direct sum Z odd}
		\Z^{(n)}\cong \Z^{(n),\mathrm{odd}} \oplus \bigoplus_{2\leq 2k\leq n}\Q(-2k)\ .
		\end{equation}
		\item A period matrix for~$\Z^{(2n+1),\mathrm{odd}}\cong\Z^{(2n+2),\mathrm{odd}}$ is 
		\begin{equation}\label{eq: period matrix Z odd} 
		\left( \begin{array}{ccccccc}
		1 & \zeta(3) & \zeta(5) & \cdots & \cdots & \zeta(2n-1) & \zeta(2n+1) \\
		&(2\pi i)^3 &  &  &  &  & \\
		 && (2\pi i)^5 &  &  & 0 & \\
		 &&  & \ddots &  &  & \\
		 &&  & & \ddots &  & \\
		 &0&  &  &  & (2\pi i)^{2n-1} & \\
		 &&  &  &  &  & (2\pi i)^{2n+1} 
		\end{array} \right)\cdot
		\end{equation}
		\end{enumerate}
		\end{prop}
		
		Proposition~\ref{prop: period matrix Z odd} implies that the odd zeta motive~$\Z^{\mathrm{odd}}$ has an infinite period matrix (\ref{eq: period matrix Z odd intro}). In particular, $\Z^{(3),\mathrm{odd}}$ is the essentially unique non-trivial extension of $\Q(-3)$ by $\Q(0)$ in $\mathsf{MT}(\mathbb{Q})$.
		
		\begin{proof}
		This is a consequence of the short exact sequence (\ref{eq: short exact sequence ZT}) and the vanishing of the extension groups $\mathrm{Ext}^1_{\mathsf{MT}(\Q)}(\Q(-2k),\Q(0))$, see (\ref{eq: extensions MTQ}).
		An alternative proof which does not use the computation of extension groups goes as follows. A basis for~$\Z_{\dR}^{(n),\mathrm{odd}}$ is given by~$v_0^{(n)}$ and the~$v_{2k+1}^{(n)}$, for~$3\leq 2k+1\leq n$, and a basis for~$\Z_{\B}^{(n),\mathrm{odd},\vee}$ is given by~$\varphi_0^{(n)}$ and the~$\varphi_{2k+1}^{(n)}$, for~$3\leq 2k+1\leq n$. This gives the desired shape for the period matrix (\ref{eq: period matrix Z odd}). Now, Euler's solution to the Basel problem implies that we have~$\zeta(2k)=\lambda_{2k}(2\pi i)^{2k}$ for every integer~$k\geq 1$, with~$\lambda_{2k}=-\frac{B_{2k}}{2(2k)!}\in\mathbb{Q}$. Thus, we may replace the basis~$(\varphi_0^{(n)},\varphi_2^{(n)},\ldots,\varphi_n^{(n)})$ of Theorem~\ref{thm: period matrix Z} by the basis~$({}'\varphi^{(n)}_0,\varphi_2^{(n)},\ldots,\varphi_n^{(n)})$ with 
		$${}'\varphi^{(n)}_0=\varphi^{(n)}_0-\sum_{2\leq 2k\leq n}\lambda_{2k}\,\varphi_{2k}^{(n)}$$
		to get a period matrix similar to (\ref{eq: period matrix Z}) where the even zeta values~$\zeta(2k)$ in the first row are replaced by~$0$. This implies the direct sum decomposition (\ref{eq: direct sum Z odd}).
		\end{proof}
		
		We finish by proving that all the objects in~$\mathsf{MT}(\Q)$ considered earlier actually live in the full subcategory~$\mathsf{MT}(\mathbb{Z})$.
	
		\begin{prop}\label{prop: Z MTZ}
		The zeta motives~$\Z^{(n)}$ and the odd zeta motives~$\Z^{(n),\mathrm{odd}}$ are objects of the category~$\mathsf{MT}(\mathbb{Z})$.
		\end{prop}
		
		\begin{proof}
		Thanks to the direct sum decomposition (\ref{eq: direct sum Z odd}), it is enough to prove it for the odd zeta motives. Let us recall the definition~\cite[D\'{e}finition 1.4]{delignegoncharov} of the category~$\mathsf{MT}(\mathbb{Z})$. According to the tannakian formalism, the de Rham realization functor~$\mathsf{MT}(\Q)\rightarrow \mathsf{grVect}_\Q$ induces an equivalence of categories
		$$\mathsf{MT}(\Q)\cong \mathsf{grRep}(\mathfrak{g}^{\mathbb{Q}}_\dR)$$
		between~$\mathsf{MT}(\Q)$ and the category of graded finite-dimensional representations of a graded Lie algebra~$\mathfrak{g}^{\mathbb{Q}}_\dR$. The degree in~$\mathfrak{g}^{\mathbb{Q}}_\dR$ is half the weight. This Lie algebra is non-positively graded. The category~$\mathsf{MT}(\mathbb{Z})$ is defined as the full subcategory of~$\mathsf{MT}(\Q)$ consisting of objects~$H$ such that the degree~$-1$ component of~$\mathfrak{g}^{\mathbb{Q}}_\dR$ acts trivially on~$H_\dR$. This is obviously the case for~$\Z^{(n),\mathrm{odd}}$, which is concentrated in weights~$0$ and~$2(2k+1)$ with~$2k+1\geq 3$ by the short exact sequence (\ref{eq: short exact sequence Z odd}).
		\end{proof}
		
		\begin{rem}
		A tannakian interpretation of the odd zeta motive goes as follows. Let~$\mathfrak{g}^{\mathbb{Z},\vee}$ be the graded dual of the fundamental Lie algebra~$\mathfrak{g}^{\mathbb{Z}}$ of the tannakian category~$\mathsf{MT}(\mathbb{Z})$. It is an ind-object in~$\mathsf{MT}(\mathbb{Z})$, independent of the choice of a fiber functor~\cite[D\'{e}finition 6.1]{delignedroiteprojective}. Then one has a short exact sequence
		$$0\rightarrow \mathbb{Q}(0) \rightarrow \mathfrak{g}^{\mathbb{Z},\vee} \rightarrow \mathfrak{u}^{\mathbb{Z},\vee}\rightarrow 0\ ,$$
		where~$\mathfrak{u}^{\mathbb{Z}}$ is the pro-unipotent radical of~$\mathfrak{g}^{\mathbb{Z}}$. One views~$\Z^{\mathrm{odd}}$ inside the exact subsequence
		$$0\rightarrow \mathbb{Q}(0) \rightarrow \Z^{\mathrm{odd}} \rightarrow \mathfrak{u}^{\mathbb{Z},\mathrm{ab},\vee} \rightarrow 0\ ,$$
		where~$\mathfrak{u}^{\mathbb{Z},\mathrm{ab},\vee} \cong \bigoplus_{k\geq 1}\Q(-(2k+1))$ is the graded dual of the abelianization of~$\mathfrak{u}^{\mathbb{Z}}$.
		\end{rem}
	
\subsection{An elementary computation of the motives~$\T^{(n)}$}\label{par: elementary}
	
		We give an elementary proof of Proposition~\ref{prop: computation T}, which only uses basic algebraic topology. The proof is Hodge-theoretic, and the only drawback is that we have to use the full faithfulness of the Hodge realization (Theorem~\ref{thm: fully faithful}). Let us consider the relative homology group
		$$\T^{(n),\vee}_\B=H_n^{\mathrm{sing}}((\mathbb{C}^*)^n,\bigcup_{1\leq i\leq n}\{x_i=1\}\cup\{x_1\cdots x_n=1\})\ .$$
		By homotopy invariance, one may replace every~$\mathbb{C}^*$ by the unit circle~$S^1=\{|x|=1\} \hookrightarrow \mathbb{C}^*$ and the divisor~$\{x_1\cdots x_n=1\}$ by its intersection with~$(S^1)^n$, and we get
		$$\T^{(n),\vee}_\B \cong H_n^{\mathrm{sing}}((S^1)^n,\bigcup_{1\leq i\leq n}\{x_i=1\}\cup\{x_1\cdots x_n=1\})\ .$$
		Let us look at the projection~$[0,1]^n \rightarrow (S^1)^n, (t_1,\ldots,t_n)\mapsto (e^{2\pi i t_1},\ldots, e^{2\pi i t_n})$. Then by excision we can write
		$$\T^{(n),\vee}_\B \cong H_n^{\mathrm{sing}}([0,1]^n, \bigcup_{1\leq i\leq n}\{t_i\in\mathbb{Z}\}\cup \{t_1+\cdots +t_n\in\mathbb{Z}\})\ .$$
		This is simply the singular homology of the unit hypercube~$[0,1]^n$ relative to the union of its faces~$\{t_i=0\}$ and~$\{t_i=1\}$, for~$1\leq i\leq n$, and the hyperplanes~$\{t_1+\cdots+t_n=k\}$ for~$k=0,1,\ldots,n$. We note that these hyperplanes cut the unit hypercube into polytopes
		$$\Delta(n,k)=\{(t_1,\ldots,t_n)\in [0,1]^n \,\, | \,\, k\leq t_1+\cdots +t_n \leq k+1\}\ ,$$
		for~$k=0,\ldots,n-1$. We note that~$\Delta(n,0)$ is the usual~$n$-simplex; the polytopes~$\Delta(n,k)$ are usually called \emph{hypersimplices}.
		
		\begin{lem}\label{lem: Delta basis}
		\begin{enumerate}
		\item The classes~$[\Delta(n,k)]$, for~$k=0,\ldots,n-1$, form a basis of~$\T^{(n),\vee}_\B$.
		\item The morphism~$j^{(n),\vee}_\B:\T^{(n),\vee}_\B\rightarrow \T^{(n-1),\vee}_\B$ sends 
			\begin{enumerate}[(a)]
			\item $[\Delta(n,0)]$ to~$[\Delta(n-1,0)]$;
			\item $[\Delta(n,k)]$ to~$[\Delta(n-1,k)]-[\Delta(n-1,k-1)]$ for~$k=1,\ldots,n-2$.
			\item $[\Delta(n,n-1)]$ to~$-[\Delta(n-1,n-2)]$.
			\end{enumerate}
		\end{enumerate}
		\end{lem}
		
		\begin{proof}
		\begin{enumerate}
		\item This is clear by excision, since collapsing the boundary of~$[0,1]^n$ and the hyperplanes~$\{t_1+\cdots+t_n=k\}$ onto a point creates a wedge sum of~$n$ spheres of dimension~$n$, one for each hypersimplex.
		\item Recall (see Remark~\ref{rem: inductive T} and \S\ref{par: inductive}) that~$j^{(n),\vee}_\B$ computes \enquote{the component of the boundary that lives on~$\{x_n=1\}$}. In the~$t$-coordinates,~$\{x_n=1\}$ corresponds to~$\{t_n=0\}$ (counted positively) and~$\{t_n=1\}$ (counted negatively). In case (b), the intersection of~$\Delta(n,k)$ with~$\{t_n=0\}$ is~$\Delta(n-1,k)$ and its intersection with~$\{t_n=1\}$ is~$\Delta(n-1,k-1)$, which proves the claim. Cases (a) and (c) are similar.
		\end{enumerate}
		\end{proof}
		
		\begin{rem}
		One may check that the sum of the classes~$[\Delta(n,k)]$, for~$k=0,\ldots,n-1$, is sent to~$0$ by the morphism~$j^{(n),\vee}_\B$. This is because this sum is represented by the unit square~$[0,1]^n$ in the~$t$-coordinates, or by the compact~$n$-torus~$(S^1)^n\subset (\C^*)^n$ in the~$x$-coordinates, which has empty boundary.
		\end{rem}
		
		The \emph{Eulerian numbers} are the coefficients of the Eulerian polynomials and are denoted by symbols~$\eulerian{n}{k}$:
		$$E_n(x)=\sum_{k=0}^{n-1}\eulerian{n}{k}\, x^k\ .$$
		They satisfy many beautiful identities, including the following recursion, which can be deduced from (\ref{eqrecurrenceE}):
		$$\eulerian{n}{k}=(n-k)\eulerian{n-1}{k-1} + (k+1)\eulerian{n-1}{k}\ .$$
		The following lemma is a classical result due to Laplace~\cite[\S 2]{foatadistributions}.
		
		\begin{lem}\label{lem: laplace}
		For~$k=0,\ldots,n-1$, the volume of the hypersimplex~$\Delta(n,k)$ is the ratio~$\dfrac{\eulerian{n}{k}}{n!}$.
		\end{lem}
		
		Recall from Remark~\ref{rem: basis TdR} that, for every integer~$n\geq 1$,~$\T^{(n)}_\dR$ has a basis~$(w_1^{(n)},\ldots,w_n^{(n)})$ which is compatible with the weight grading and with the morphisms~$j^{(n)}_\dR:\T^{(n-1)}_\dR\rightarrow \T^{(n)}_\dR$. We let~$P_n$ be the period matrix of~$\T^{(n)}$ with respect to the~$w$-basis and the~$\Delta$-basis from Lemma~\ref{lem: Delta basis}. The first period matrix~$P_1$ is simply the~$1\times 1$ matrix~$(2\pi i)$. Let us introduce the following~$n\times n$ integer matrix encoding the family of Eulerian numbers:
		$$
		A_n=\left( \begin{array}{ccccccc}
		1 &  &  &  &  &  & \eulerian{n}{0} \\
		-1 & 1 &  &  & 0 &  & \eulerian{n}{1} \\
		 & -1&1 &  &  &  & \eulerian{n}{2} \\
		 &&\ddots  & \ddots &  &  & \\
		 &&  & \ddots & \ddots &  & \\
		 &0&&  &  -1 & 1 & \eulerian{n}{n-2}  \\
		 &&  &  &  &  -1& \eulerian{n}{n-1} 
		\end{array} \right)\cdot
		$$
		
		
		\begin{prop}\label{prop: recurrence P}
		The period matrices~$P_n$ satisfy the recurrence relation 
		$$P_n=A_n\left( \begin{array}{ccccc|c}
		   &  &  &  &  & 0 \\
		  & &   &  &  & 0 \\
		 &  & P_{n-1} &  &  & \vdots \\
		 &  &  &  &  & \\
		 &&  &  &  & 0  \\ \hline
		  0& 0  &\hdots  &  &  0 & \dfrac{(2\pi i)^n}{n!} 
		\end{array} \right)\cdot$$
		\end{prop}
		
		\begin{proof}
		Recall the short exact sequence (\ref{eq: short exact sequence T})
		$$0\rightarrow \T^{(n-1)} \stackrel{j^{(n)}}{\longrightarrow} \T^{(n)} \rightarrow H^n(T^n)\rightarrow 0$$
		and the fact (see Remark~\ref{rem: basis TdR}) that the morphism~$j^{(n)}$ is compatible with the~$w$-bases. Then Lemma~\ref{lem: Delta basis} shows that the first~$(n-1)$ columns of~$P_n$ are as stated. It only remains to compute the entries in the last column, i.e., compute the integral of the~$n$-form~$\frac{dx_1}{x_1}\wedge \cdots \wedge \frac{dx_n}{x_n}$ on a hypersimplex~$\Delta(n,k)$. After the change of variables~$(x_1,\ldots,x_n)=(e^{2\pi i t_1},\ldots,e^{2\pi i t_n})$, one sees that this integral is simply~$(2\pi i)^n$ times the volume of~$\Delta(n,k)$, and completes the proof thanks to Lemma~\ref{lem: laplace}.
		\end{proof}
		
		We note that the period matrices~$P_n$ are not block upper-triangular. This is because the~$\Delta$-basis is not compatible with the weight filtration. We thus have to introduce a change of basis. Let~$(Q_n)_{n\geq 1}$ be the family of matrices (with rational entries) defined by~$Q_1=(1)$ and the recurrence relation
		$$Q_n=\left( \begin{array}{ccccc|c}
		   &  &  &  &  & 0 \\
		  & &   &  &  & 0 \\
		 &  & Q_{n-1} &  &  & \vdots \\
		 &  &  &  &  & \\
		 &&  &  &  & 0  \\ \hline
		  0& 0  &\hdots  &  &  0 & n!
		\end{array} \right) A_n^{-1}\ .$$
		
		The first terms are 
		$$Q_1=\left(\begin{array}{c} 1 \end{array}\right)  ,  Q_2= \left(\begin{array}{cc} \frac{1}{2} & -\frac{1}{2} \\ 1 & 1 \end{array}\right) , Q_3= \left(\begin{array}{ccc} \frac{1}{3} & -\frac{1}{6} & \frac{1}{3} \\ 1 & 0 & -1 \\ 1 & 1 & 1 \end{array}\right)  ,  Q_4= \left(\begin{array}{cccc} \frac{1}{4} & -\frac{1}{12} & \frac{1}{12} & -\frac{1}{4}\\ \frac{11}{12} & -\frac{1}{12} & -\frac{1}{12} & \frac{11}{12} \\ \frac{3}{2} & \frac{1}{2} & -\frac{1}{2} & -\frac{3}{2} \\ 1 & 1 & 1 & 1 \end{array}\right) .$$		
		
		Let us put
		$$ \begin{pmatrix} \Sigma_1^{(n)}\\ \Sigma_2^{(n)}\\ \vdots\\ \Sigma_n^{(n)} \end{pmatrix} = Q_n \begin{pmatrix} \Delta(n,0)\\ \Delta(n,1)\\ \vdots\\ \Delta(n,n-1) \end{pmatrix}\ .$$
		We view~$\Sigma_k^{(n)}$ as a relative cycle with rational coefficients. The change of indexing is here to remind the reader that~$\Sigma_k^{(n)}$ lives in weight less than or equal to $2k$. We have thus proved the following result.
		
		\begin{prop}\label{prop: Sigma basis}
		The classes~$[\Sigma_k^{(n)}]$, for~$k=1,\ldots,n$, form a basis of~$\T^{(n),\vee}_\B$ and the period matrix of~$\T^{(n)}$ in the~$w$-basis and the~$\Sigma$-basis is the diagonal matrix~$\mathrm{Diag}(2\pi i, \ldots, (2\pi i)^n)$.
		\end{prop}
		
		\begin{proof}
		This amounts to saying that the product~$Q_nP_n$ is the matrix~$\mathrm{Diag}(2\pi i, \ldots, (2\pi i)^n)$, which is easily proved by induction on~$n$ using Proposition~\ref{prop: recurrence P}.
		\end{proof}
		
		By using Theorem~\ref{thm: fully faithful}, we thus get an alternate (Hodge-theoretic) proof of Proposition~\ref{prop: computation T}. 
		
		\begin{rem}\label{rem: Sigma Psi}
		Proposition~\ref{prop: Sigma basis} implies that we can choose~$(\Sigma^{(n-1)}_1,\ldots,\Sigma^{(n-1)}_{n-1})$ as representatives for the classes~$(\psi^{(n-1)}_1,\ldots,\psi^{(n-1)}_{n-1})$ from the proof of Theorem~\ref{thm: period matrix Z}.
		\end{rem}
		
		\begin{rem}\label{rem: Sigma torus}
		One can easily prove that the last row of the matrix~$Q_n$ is filled with~$1$s, which means that~$\Sigma_n^{(n)}$ is homologous to the unit hypercube~$[0,1]^n$. In the~$x$-coordinates, it is homologous to the compact~$n$-torus~$(S^1)^n\subset (\C^*)^n$.
		\end{rem}

\section{Linear forms in zeta values}\label{section6}

	We apply our results from the previous section to prove Theorems~\ref{thm: coefficients intro} and~\ref{thm: vanishing intro} from the Introduction.
	
	\subsection{Integral formulas for the coefficients}
	
		\begin{thm}\label{thm: coefficients}
		For~$\omega$ an integrable algebraic differential form on~$X_n-A_n$, we have
		\begin{equation}\label{eq: integral coefficients a}
		\int_{[0,1]^n}\omega=a_0(\omega)+a_2(\omega)\zeta(2)+\cdots+a_n(\omega)\zeta(n)
		\end{equation}
		with~$a_k(\omega)$ a rational number for every~$k$, given for~$k=2,\ldots,n$ by the formula
		\begin{equation}\label{eq: coefficients a sigma}
		a_k(\omega)=(2\pi i)^{-k}\,\langle\varphi_k^{(n)},[\omega]\rangle \ .
		\end{equation}
		\end{thm}
		
		\begin{proof}
		According to Proposition~\ref{prop: poles}, the class~$[\omega]$ defines an element in~$\Z_{n,\dR}$, hence we may write
		$$[\omega]=a_0(\omega)v_0+a_2(\omega)v_2+\cdots+a_n(\omega)v_n$$
		with~$a_k(\omega)\in\Q$ for every~$k$. Pairing with the class~$\varphi_0^{(n)}=[\square^n]$ gives the equality (\ref{eq: integral coefficients a}), and pairing with the class~$\varphi_k^{(n)}$,~$k=2,\ldots,n$, gives the equality (\ref{eq: coefficients a sigma}).
		\end{proof}
		
		\begin{rem}\label{rem: tubular}
		If we represent the class~$\varphi_k^{(n)}$ by a relative cycle~$\sigma_k^{(n)}$, then (\ref{eq: coefficients a sigma}) becomes 
		$$a_k(\omega)=(2\pi i)^{-k}\int_{\sigma_k^{(n)}}\omega\ .$$
		
		Here we will not give explicit representatives for the classes~$\varphi_k^{(n)}$. Recall from the proof of Theorem~\ref{thm: period matrix Z} that the class~$\varphi_k^{(n)}$ is the image by the map~$p_\B^{(n),\vee}:\T^{(n-1),\vee}_\B\rightarrow \Z^{(n),\vee}_\B$ of an element~$\psi_{k-1}^{(n-1)}$, which by Remark~\ref{rem: Sigma Psi} can be represented by the cycle~$\Sigma_{k-1}^{(n-1)}$. The question is then how to compute the map~$p_B^{(n),\vee}$ at the level of cycles. Such a task would involve the following ingredients. Let~$T\subset \C^n$ be a tubular neighborhood of~$A_n(\C)$ in~$\C^n$. Let us denote by~$\rho:T\rightarrow A_n(\C)$ the corresponding projection, and by~$\partial\rho:\partial T\rightarrow A_n(\C)$ the projection corresponding to the boundary of the tubular neighborhood; it is an~$S^1$-bundle. The natural map~$H^{\mathrm{sing}}_r(A_n(\C))\rightarrow H^{\mathrm{sing}}_{r+1}(\C^n-A_n(\C))$ can be computed at the level of singular chains by mapping an~$r$-cycle~$\sigma$ to the~$(r+1)$-cycle~$(\partial\rho)^{-1}(\sigma)$. We note that since~$A_n(\C)$ does not intersect the hyperplanes~$\{x_i=0\}$, we can do the computation with a tubular neighborhood inside~$(\C^*)^n$ and get representatives in~$(\C^*)^n$. Now if we want to play this game for the relative homology groups~$\Z^{(n),\vee}_\B$, we need the tubular neighborhood to be \enquote{compatible} with the subvariety~$B_n(\C)$, in the sense that~$\rho$ should pull back~$A_n(\C)\cap B_n(\C)$ to~$B_n(\C)$. At this point, it is probably easier to ask for something weaker than a tubular neighborhood, i.e., something that is a tubular neighborhood on a dense open subset of~$A_n(\C)$ (this does not change anything for the integral formulas). We will not try to give formulas here and postpone this discussion to a future article. Nevertheless, we can give more explicit formulas than (\ref{eq: coefficients a sigma}) in two situations.
		\end{rem}
		
		\subsubsection{The highest weight coefficient}

		Let us fix real numbers~$\rho_1,\ldots,\rho_{n-1},\rho_n>0$ and let us introduce the cycle~$S^{(n)}\subset \C^n-A_n(\C)$ defined by the conditions
		$$|x_1|=\rho_1, \ldots, \, |x_{n-1}|=\rho_{n-1}, \, \left|x_n-\frac{1}{x_1\cdots x_{n-1}}\right|=\rho_n\ .$$
		
		\begin{prop}
		Let~$\omega$ be an integrable differential form on~$X_n-A_n$. Then the highest weight coefficient~$a_n(\omega)$ from Theorem~\ref{thm: coefficients} is given by the integral formula
		$$a_n(\omega)=(2\pi i)^{-n}\, \int_{S^{(n)}}\omega\ .$$
		\end{prop}		
		
		\begin{proof}
		The integral formula is obviously independent of the choice of~$\rho_1,\ldots,\rho_{n-1},\rho_n$ and we can assume that we have~$\rho_1=\cdots=\rho_{n-1}=\rho_n=1$. We have noted in Remark~\ref{rem: Sigma torus} that the highest weight basis vector~$\psi^{(n-1)}_{n-1}$ of~$\T^{(n-1),\vee}_\B$ can be represented by the~$(n-1)$-torus~$\{|x_1|=\cdots=|x_{n-1}|=1\}$. Since this has an empty boundary we can make the computation explained in Remark~\ref{rem: tubular} with the choice of \emph{any} tubular neighborhood of~$A_n(\C)$ in~$\C^n$, for instance the one defined by~$\left|x_n-\frac{1}{x_1\cdots x_{n-1}}\right|\leq 1$, with projection map~$\rho(x_1,\ldots,x_n)=(x_1,\ldots,x_{n-1},\frac{1}{x_1\cdots x_{n-1}})$. The pullback of the~$(n-1)$-torus by the projection~$\partial\rho$ is exactly~$S^{(n)}$.
		\end{proof}
		
		The case~$n=2$ is Rhin and Viola's contour integral for~$\zeta(2)$~\cite[Lemma 2.6]{rhinviolazeta2}.
		
		\subsubsection{The case of forms with simple poles}
		
		We say that a differential form on~$X_n-A_n$ has a \emph{simple pole} along~$A_n$ if it can be written as
		$$\omega=\alpha+\mathrm{dlog}(1-x_1\cdots x_n)\wedge \beta\ ,$$
		where~$\alpha$ and~$\beta$ do not have poles along~$A_n$. The \emph{residue} of such a form along~$A_n$ is the restriction
		$$\mathrm{Res}(\omega)=\beta_{|A_n}\ .$$
		Recall that the relative cycles~$\Sigma_{k-1}^{(n-1)}$ were defined in \S\ref{par: elementary}.
		
		\begin{prop}
		Let~$\omega$ be an integrable differential form on~$X_n-A_n$ which has a simple pole along~$A_n$. Then the coefficients~$a_k(\omega)$,~$k=2,\ldots,n$, from Theorem~\ref{thm: coefficients} are given by the integral formulas
		$$a_k(\omega)=(2\pi i)^{-k+1}\,\int_{\Sigma^{(n-1)}_{k-1}}\mathrm{Res}(\omega)\ .$$
		\end{prop}
		
		\begin{proof}
		Recall from the proof of Theorem~\ref{thm: period matrix Z} that we have defined
		$$\varphi^{(n)}_k=p^{(n),\vee}_\B(\psi_{k-1}^{(n-1)})\ ,$$ 
		where~$(\psi^{(n-1)}_1,\ldots,\psi^{(n-1)}_{n-1})$ is a basis of~$\T^{(n-1),\vee}_\B$ for which the period matrix is diagonal. In the light of Remark~\ref{rem: Sigma Psi} we see that~$\psi_{k-1}^{(n-1)}$ is the class of the cycle~$\Sigma_{k-1}^{(n-1)}$, hence we get
		$$a_k(\omega)=(2\pi i)^{-k}\, \langle\, p^{(n),\vee}_\B([\Sigma_{k-1}^{(n-1)}])\, , \, [\omega] \, \rangle = (2\pi i)^{-k+1}\,\langle\,[\Sigma_{k-1}^{(n-1)}]\,,\,p^{(n)}_\dR([\omega])\,\rangle\ ,$$
		where the extra~$2\pi i$ comes from the Tate twist at the target of~$p^{(n)}$. Since~$\omega$ has a simple pole,~$p^{(n)}_\dR([\omega])$ is simply the class of~$\mathrm{Res}(\omega)$, hence the result.
		\end{proof}
		
%
%
%
%
		
	\subsubsection{Vanishing of coefficients}

		\begin{thm}\label{thm: vanishing}
		For~$\omega$ an integrable algebraic differential form on~$X_n-A_n$, we have:
		\begin{enumerate}
		\item if~$\tau.\,\omega=\omega$ then~$a_k(\omega)=0$ for~$k\neq 0$ even;
		\item if~$\tau.\,\omega=-\omega$ then~$a_k(\omega)=0$ for~$k$ odd.
		\end{enumerate}
		\end{thm}

		\begin{proof}
		Let us assume that we have~$\tau.\omega=\omega$, and let us write~$x$ for the image of~$[\omega]$ in~$\T_\dR^{(n-1)}$. Then we have~$\tau.x=x$; according to Proposition~\ref{prop: computation T}, this implies that~$x$ only has components of weights~$2k$ with~$k$ even. Thus,~$[\omega]\in\Z_\dR^{(n)}$ only has components in weight~$0$ and~$2k$ with~$k$ odd, which implies that we have~$a_k(\omega)=0$ for~$k\neq 0$ even. The second case is similar.
		\end{proof}
		
		Let us write an integrable form as
		\begin{equation}\label{eq: integrable form end}
		\omega=\dfrac{P(x_1,\ldots,x_n)}{(1-x_1\cdots x_n)^N}\,dx_1\cdots dx_n~
		\end{equation}
		with~$P(x_1,\ldots,x_n)$ a polynomial with rational coefficients and~$N\geq 0$ an integer. Then we have
		\begin{equation}\label{eq: invariance equivalence}
		\tau.\omega=\pm\omega \;\;\Leftrightarrow\;\; P(x_1,\ldots,x_n)=\pm(-1)^{N+n}(x_1\cdots x_n)^{N-2}P(x_1^{-1},\ldots,x_n^{-1})\ .
		\end{equation}
		
	
	\subsection{The Ball--Rivoal integrals}
	
		We apply Theorems~\ref{thm: coefficients} and~\ref{thm: vanishing} to a special family of integrals.
	
		\begin{coro}\label{coro: ballrivoal}
		Let~$u_1,\ldots,u_n,v_1,\ldots,v_n\geq 1$ and~$N\geq 0$ be integers such that~$v_1+\cdots+v_n\geq N+1$. Then the integral
		\begin{equation}\label{eq: integrals BR}
		\int_{[0,1]^n}\dfrac{x_1^{u_1-1}\cdots x_n^{u_n-1}(1-x_1)^{v_1-1}\cdots (1-x_n)^{v_n-1}}{(1-x_1\cdots x_n)^N}\,dx_1\cdots dx_n
		\end{equation}
		is absolutely convergent and evaluates to a linear combination
		$$a_0+a_2\zeta(2)+a_3\zeta(3)+\cdots+a_n\zeta(n)$$
		with~$a_k$ a rational number for every~$k$.
		If furthermore we have~$2u_i+v_i=N+1$ for every~$i$, then we get:
		\begin{enumerate}
		\item if~$(n+1)(N+1)$ is odd then~$a_k=0$ for~$k\neq 0$ even;
		\item if~$(n+1)(N+1)$ is even then~$a_k=0$ for~$k$ odd.
		\end{enumerate}
		\end{coro}
		
		\begin{proof}
		This is a direct application of Theorem~\ref{thm: vanishing}. The polynomial
		$$P(x_1,\ldots,x_n)=x_1^{u_1-1}\cdots x_n^{u_n-1}(1-x_1)^{v_1-1}\cdots (1-x_n)^{v_n-1}$$ 
		satisfies
		$$P(x_1,\ldots,x_n)=(-1)^{n+v_1+\cdots+v_n}x_1^{2u_1+v_1-3}\cdots x_n^{2u_n+v_n-3}P(x_1^{-1},\ldots,x_n^{-1})\ .$$
		Let us assume that we have~$2u_i+v_i=N+1$ for every~$i$, then~$v_1+\cdots+v_n \equiv n(N+1) \;(\mathrm{mod}\; 2)$ and we get
		$$P(x_1,\ldots,x_n)=-(-1)^{(n+1)(N+1)}(-1)^{N+n}(x_1\cdots x_n)^{N-2}P(x_1^{-1},\ldots,x_n^{-1})\ ,$$
		hence the result, in view of (\ref{eq: invariance equivalence}).
		\end{proof}
		
		Corollary~\ref{coro: ballrivoal} applies in particular to the special case
		$$N=(2r+1)m+2, \; u_i=rm+1, \; v_i=m+1$$ 
		for some integer parameters~$r,m\geq 0$ satisfying~$n(m+1)\geq (2r+1)m+3$. We then recover the integrals considered by Ball and Rivoal~\cite[Lemme 2]{ballrivoal}. The vanishing of the coefficients is~\cite[Lemme 1]{ballrivoal}. The notations~$(a,n,r)$ in~\cite{ballrivoal} correspond to our notations~$(n-1,m,r)$.\\
		
		The integrals (\ref{eq: integrals BR}) can be expressed as generalized hypergeometric series
		$$\left(\prod_{i=1}^n \dfrac{(u_i-1)!(v_i-1)!}{(u_i+v_i-1)!}\right) {}_{n+1}F_n \!\left ( \begin{matrix} {u_1,\ldots,u_n,N}\\ {u_1+v_1,\ldots,u_n+v_n}\end{matrix} \,;\, {\displaystyle 1}\right ) \hspace{5cm}$$
		\begin{equation}\label{eq: hypergeometric}
		\hspace{3cm}=\,\dfrac{\prod_{i=1}^n(v_i-1)!}{(N-1)!}\,\,\sum_{k\geq 0}\dfrac{(k)_{u_1}\cdots (k)_{u_n}(k+1)_{N-1}}{(k)_{u_1+v_1}\cdots (k)_{u_n+v_n}}\cdot
		\end{equation}
		If~$2u_i+v_i=N+1$ then the corresponding generalized hypergeometric series is said to be \emph{well-poised}.
		
	\subsection{Weight drop}
	
		In the context of Theorem~\ref{thm: coefficients}, we say that the integral~$\int_{[0,1]^n}\omega$ has \emph{weight drop} if the highest weight coefficient~$a_n(\omega)$ vanishes. This amounts to saying that the class~$[\omega]$ actually lives in the step~$W_{2(n-1)}\Z_{n,\dR}$ of the weight filtration, hence the terminology. We give a sufficient condition for this phenomenon to happen.
		
		\begin{lem}\label{lem: integration}
		Let~$u,v\geq 1$ and~$N\geq 0$ be integers such that~$u+v\leq N$. Then there exists a polynomial~$P(t)$ with rational coefficients such that
		$$\int_0^1 \dfrac{x^{u-1}(1-x)^{v-1}}{(1-tx)^N}\,dx=\dfrac{P(t)}{(1-t)^{N-v}}$$
		for every~$0\leq t<1$.
		\end{lem}
		
		\begin{proof}
		We can write 
		$$x^{u-1}(1-x)^{v-1}=\sum_{k=0}^{u+v-2}a_k(t)(1-tx)^k$$
		with~$a_k(t)$ a Laurent polynomial with rational coefficients for every~$k$. We then have
		$$\dfrac{x^{u-1}(1-x)^{v-1}}{(1-tx)^N}=\sum_{k=0}^{u+v-2}\frac{a_k(t)}{(1-tx)^{N-k}}$$
		and all the powers of~$(1-tx)$ appearing in the denominators are greater than or equal to $N-(u+v-2)\geq N-u-v+2\geq 2$. Thus, we may integrate and get
		$$\int_0^1 \dfrac{x^{u-1}(1-x)^{v-1}}{(1-tx)^N}dx=\frac{Q(t)}{(1-t)^{N-1}}$$
		with~$Q(t)$ a Laurent polynomial with rational coefficients. The left-hand side has a limit when~$t$ tends to~$0$, so~$Q(t)$ has to be a polynomial. To finish, it is enough to show that 
		$$(1-t)^{N-v}\int_0^1 \dfrac{x^{u-1}(1-x)^{v-1}}{(1-tx)^N}\,dx$$
		is bounded when~$t$ approaches~$1$. We make the change of variables~$s=1-t$,~$y=1-x$, and consider integrals
		$$s^{N-v}\int_0^1 \dfrac{(1-y)^{u-1}y^{v-1}}{(y+s-ys)^N}\,dy$$
		with~$s$ approaching~$0$. Since~$(1-y)^{u-1}\leq 1$ and~$y+s-ys\geq \frac{1}{2}(y+s)$, it is enough to prove that the quantities
		$$s^{N-v}\int_0^1 \dfrac{y^{v-1}}{(y+s)^N}\,dy$$
		are bounded when~$s$ approaches~$0$. This equals
		$$s^{N-v}\int_0^1\left(\dfrac{y}{y+s}\right)^{v-1}\dfrac{dy}{(y+s)^{N-v+1}}\leq s^{N-v}\int_0^1\dfrac{dy}{(y+s)^{N-v+1}}= \dfrac{1}{N-v}\left(1-\left(\dfrac{s}{1+s}\right)^{N-v}\right)$$
		and we are done.
		\end{proof}
		
		\begin{prop}\label{prop: weight drop}
		Let~$u_1,\ldots,u_n,v_1,\ldots,v_n\geq 1$ and~$N\geq 0$ be integers such that~$v_1+\cdots+v_n\geq N+1$. Let us assume that there exists an index~$i\in\{1,\ldots, N\}$ such that 
		$$u_i+v_i\leq N\ .$$		
		Then the integral
		$$\int_{[0,1]^n}\dfrac{x_1^{u_1-1}\cdots x_n^{u_n-1}(1-x_1)^{v_1-1}\cdots (1-x_n)^{v_n-1}}{(1-x_1\cdots x_n)^N}\,dx_1\cdots dx_n$$
		is absolutely convergent and evaluates to a linear combination
		$$a_0+a_2\zeta(2)+a_3\zeta(3)+\cdots+a_{n-1}\zeta(n-1)$$
		with~$a_i\in\Q$ for every~$i$.
		\end{prop}
		
		\begin{proof}
		By symmetry, we can assume that~$u_n+v_n\leq N$. Therefore, applying Lemma~\ref{lem: integration} to the variables~$x=x_n$ and~$t=x_1\cdots x_{n-1}$ in the integral leads to the~$(n-1)$-dimensional integral
		$$\int_{[0,1]^{n-1}}\dfrac{x_1^{u_1-1}\cdots x_{n-1}^{u_{n-1}-1}(1-x_1)^{v_1-1}\cdots (1-x_{n-1})^{v_{n-1}-1}P(x_1\cdots x_{n-1})}{(1-x_1\cdots x_{n-1})^{N-v_n}}\,dx_1\cdots dx_{n-1}\ .$$
		Since~$v_1+\cdots+v_{n-1}\geq N-v_n+1$, one can then finish thanks to Theorem~\ref{thm: coefficients}.
		\end{proof}
		
		Note that Proposition~\ref{prop: weight drop} applies in particular if for every~$i$,~$2u_i+v_i=N+1$. This gives in particular a geometric interpretation of the weight drop in the Ball--Rivoal integrals~\cite{rivoalcras,ballrivoal}, which comes from the representations as hypergeometric series (\ref{eq: hypergeometric}). Note that a careful analysis of the degree of the polynomial~$P(t)$ in Lemma~\ref{lem: integration} can lead to sufficient conditions for the vanishing of the subleading coefficients.

\appendix

\section{An approach via series (joint with Don Zagier)}

	The aim of this appendix is to give an elementary construction of the coefficients~$a_k(\omega)$ from Theorem~\ref{thm: coefficients}. The dictionary between integrals and sums of series leads to an interpretation of the (de Rham realization of the) zeta motive~$\mathcal{Z}$, modulo weight~$0$, in terms of rational functions in one variable.

	\subsection{Series, integrals, and zeta values}
	
		\subsubsection{Series of rational functions and zeta values}
	
			We denote by~$\Q(k)$ the field of rational functions in the variable~$k$ with rational coefficients. Let~$V$ denote the subspace of~$\mathbb{Q}(k)$ consisting of rational functions with poles in~$\{-1,-2,-3,\ldots\}$ and~$V_0$ be the subspace of functions vanishing at~$\infty$. Then~$V=V_0\oplus\mathbb{Q}[k]$ and the set of functions~$(k+j)^{-r}$, with~$j,r\geq 1$ integers, is a basis of~$V_0$. The forward difference operator~$\Delta:\mathbb{Q}(k)\rightarrow\mathbb{Q}(k)$ defined by~$\Delta R(k)=R(k+1)-R(k)$ preserves the spaces~$V$ and~$V_0$ and one has direct sum decompositions~$V=\Delta(V)\oplus B$ and~$V_0=\Delta(V_0)\oplus B$, where~$B$ is the space spanned by the functions~$(k+1)^{-r}$, for~$r\geq 1$ integers. We thus have an identification~$V_0/\Delta(V_0)\cong V/\Delta(V)$ and an isomorphism
			
			\begin{equation}\label{eq:appendixbeta}
			\beta:V/\Delta(V) \stackrel{\simeq}{\longrightarrow} \bigoplus_{r\geq 1}\mathbb{Q} \; \; , \;\; R \mapsto (\beta_1(R),\beta_2(R),\ldots)\ ,
			\end{equation}
			where the numbers~$\beta_r(R)\in\mathbb{Q}$, for~$R\in V$, are defined by 
			$$R(k) \equiv \sum_{r\geq 1} \dfrac{\beta_r(R)}{(k+1)^r} \;\; (\mathrm{mod}\;\Delta(V))\ .$$	
			
			For~$R\in V_0$ we can write
			$$R(k)=\sum_{r\geq 1} \dfrac{\beta_r(R)}{(k+1)^r} - \Delta R_0(k)\ ,$$
			for some~$R_0\in V_0$, which is unique because~$\Delta:V_0\rightarrow V_0$ is injective. Thus, the sum~$\sum_{k= 0}^\infty R(k)$ is absolutely convergent if and only if~$R\in V_0$ and~$\beta_1(R)=0$, and in this case we have 
			\begin{equation}\label{eq:appendixsumR}
			\sum_{k=0}^\infty R(k) = R_0(0) + \sum_{r\geq 2} \beta_r(R)\,\zeta(r) \;\;\;\;\; \in \; \mathbb{Q}\;+\: \sum_{r\geq 2} \mathbb{Q} \,\zeta(r)\ .
			\end{equation}
			
		\subsubsection{From differential forms to rational functions}
		
			For~$n\geq 1$ an integer, we define
			$$\Omega_n=\mathbb{Q}[x_1,\ldots,x_n,(1-x_1\cdots x_n)^{-1}]$$ 
			and we interpret an element~$F\in\Omega_n$ as the algebraic differential~$n$-form~$\omega=F\, dx_1\cdots dx_n$.
			
			\begin{lem}
			The formula
			\begin{equation}\label{eq:appendixphi}
			\Phi_n\left(\dfrac{x_1^{a_1-1}\cdots x_n^{a_n-1}}{(1-x_1\cdots x_n)^N}\right)= \left\{ \begin{array}{cl} 0 & \textnormal{if } N=0  \\ \displaystyle\binom{k+N -1}{N-1}\dfrac{1}{(k+a_1)\cdots (k+a_n)} & \textnormal{if } N\geq 1  \end{array} \right. ,
			\end{equation}
			for~$a_1,\ldots,a_n\geq 1$ and~$N\geq 0$ integers, defines a linear map~$\Phi_n:\Omega_n\rightarrow V/\Delta(V)$.
			\end{lem}
			
			\begin{proof}
			If we rewrite~$x_1^{a_1-1}\cdots x_n^{a_n-1}$ as~$\dfrac{x_1^{a_1-1}\cdots x_n^{a_n-1}-x_1^{a_1}\cdots x_n^{a_n}}{1-x_1\cdots x_n}$, then its image by~$\Phi_n$ is 
			\begin{eqnarray*}
			\dfrac{1}{(k+a_1)\cdots (k+a_n)}-\dfrac{1}{(k+a_1+1)\cdots (k+a_n+1)}  & = & \Delta\left(-\dfrac{1}{(k+a_1)\cdots (k+a_n)}\right) \\
			& \equiv & 0 \;\; (\mathrm{mod}\;\Delta(V))\ .
			\end{eqnarray*}
			 For~$N\geq 1$, if we rewrite~$\dfrac{x_1^{a_1-1}\cdots x_n^{a_n-1}}{(1-x_1\cdots x_n)^N}$ as~$\dfrac{x_1^{a_1-1}\cdots x_n^{a_n-1}-x_1^{a_1}\cdots x_n^{a_n}}{(1-x_1\cdots x_n)^{N+1}}$, then we replace the function~$R(k)=\displaystyle\binom{k+N -1}{N-1}\dfrac{1}{(k+a_1)\cdots (k+a_n)}$ by the function
			\begin{eqnarray*}
			R^*(k) & = & \displaystyle\binom{k+N}{N}\left(\dfrac{1}{(k+a_1)\cdots (k+a_n)}-\dfrac{1}{(k+a_1+1)\cdots (k+a_n+1)}\right)  \\
			 & = & R(k) +\dfrac{kR(k)-(k+1)R(k+1)}{N} \\
			 & \equiv & R(k) \;\; (\mathrm{mod}\;\Delta(V))\ .
			\end{eqnarray*}
			This shows that the definition of ~$\Phi_n(F)$  for~$F \in (1-x_1\cdots x_n)^{-N} \mathbb{Q}[x_1,\ldots,x_n]$ is independent of the choice of~$N$.
			\end{proof}
			
			Combining with (\ref{eq:appendixbeta}), we get well-defined maps 
			$$b_r:\Omega_n \rightarrow \mathbb{Q} \;\; , \;\; \omega \mapsto \beta_r(\Phi_n(\omega))\ .$$
			Note that this is zero for~$r>n$ for degree reasons. We denote by~$\Omega_n^{\mathrm{int}}\subset \Omega_n$ the subspace of integrable differential forms, which are the forms~$\omega$ such that the integral~$\int_{[0,1]^n}\omega$ is absolutely convergent (see Definition~\ref{def: integrable} and Propositions~\ref{prop: poles} and~\ref{prop: polesconverse}).
			
			\begin{prop}\label{prop:appendixb1}
			For every~$\omega\in\Omega_n^{\mathrm{int}}$ we have~$b_1(\omega)=0$ and
			$$\int_{[0,1]^n}\omega = b_2(\omega)\zeta(2)+\cdots+b_n(\omega)\zeta(n) \;\; (\mathrm{mod}\;\mathbb{Q})\ .$$
			\end{prop}
			
			\begin{proof}
			Let us write~$\omega=\dfrac{P(x_1,\ldots,x_n)}{(1-x_1\cdots x_n)^N}$ with~$P(x_1,\ldots,x_n)$ a polynomial with rational coefficients and~\hbox{$N\geq 1$} an integer. Let~$R\in V$ be the representative of~$\Phi_n(\omega)$ obtained by applying (\ref{eq:appendixphi}) to every monomial in~$P(x_1,\ldots,x_n)$ and using linearity. Then the formula 
			$$\dfrac{1}{(1-x)^N}=\sum_{k=0}^\infty \binom{k+N-1}{N-1}x^k~$$
			implies that we have 
			$$\int_{[0,1]^n}\omega = \sum_{k=0}^\infty R(k)\ .$$
			Thus, the sum~$\sum_{k=0}^\infty R(k)$ is convergent, which implies that we have~$R\in V_0$ and~$\beta_1(R)=0$. The claim then follows from (\ref{eq:appendixsumR}).
			\end{proof}
			
			Proposition~\ref{prop:appendixb1} implies that there is a well-defined map~$b_0:\Omega_n^{\mathrm{int}}\rightarrow\mathbb{Q}$ such that for every~$\omega\in\Omega_n^{\mathrm{int}}$ we have
			\begin{equation}\label{eq:appendixintegralb}
			\int_{[0,1]^n}\omega = b_0(\omega)+b_2(\omega)\zeta(2)+\cdots+b_n(\omega)\zeta(n)\ .
			\end{equation}
			
			We note that applying~$\Phi_n$ to the integrals (\ref{eq: integrals BR}) leads to the hypergeometric series representations (\ref{eq: hypergeometric}).
			
		\subsubsection{Parity}
		
			Let us recall that~$\tau$ denotes the involution~$(x_1,\ldots,x_n)\mapsto (x_1^{-1},\ldots,x_n^{-1})$. The following proposition is nothing but a generalization of the classical well-poised symmetry of the hypergeometric series~(\ref{eq: hypergeometric}), and is similar to the parity considerations in~\cite[\S 8]{zudilinoddzetavalues} and~\cite[\S 3.1]{cressonfischlerrivoalseries}.
		
			\begin{prop}
			Let~$\omega\in\Omega_n$ be a differential form such that~$\tau.\omega$ belongs to~$\Omega_n$. We have, for every integer~$r\geq 1$,
			$$b_r(\tau.\omega)=(-1)^{r-1}b_r(\omega)\ .$$
			In particular, we have:
			\begin{enumerate}
			\item if~$\tau.\,\omega=\omega$ then~$b_r(\omega)=0$ for~$r\neq 0$ even;
			\item if~$\tau.\,\omega=-\omega$ then~$b_r(\omega)=0$ for~$r$ odd.
			\end{enumerate}
			\end{prop}
			
			\begin{proof}
			Let~$R$ and~$S$ be representatives of~$\Phi_n(\omega)$ and~$\Phi_n(\tau.\omega)$ respectively, constructed as in the proof of Proposition~\ref{prop:appendixb1}. The involution~$\tau$ acts on differential forms by the formula 
			$$\dfrac{x_1^{a_1-1}\cdots x_n^{a_n-1}}{(1-x_1\cdots x_n)^N}\, dx_1\cdots dx_n \;\; \mapsto \;\; (-1)^{N+n}\,\dfrac{x_1^{N-a_1-1}\cdots x_n^{N-a_n-1}}{(1-x_1\cdots x_n)^N}\, dx_1\cdots dx_n\ .$$ 
			Thus, by looking at the formula for~$\Phi_n$, we see that we have~$S(k)=-R(-N-k)$. This implies, for every integer~$r\geq 1$, the equality,
			$$\beta_r(S)=(-1)^{r-1}\beta_r(R)\ ,$$
			and the claim follows.
			\end{proof}

	\subsection{Comparison of the coefficients}

			The aim of this section is to prove the following theorem.	
	
			\begin{thm}\label{thmappendix}
			For every~$\omega\in\Omega_n^{\mathrm{int}}$ and every integer~$r=0,2,\ldots,n$ we have~$a_r(\omega)=b_r(\omega)$.
			\end{thm}
			
			Note that this theorem would follow from the conjecture that~$1$ and the zeta values~$\zeta(n),\, n\geq 2$, are linearly independent over~$\mathbb{Q}$, by looking at equations (\ref{eq: integral coefficients a}) and (\ref{eq:appendixintegralb}).
			
		\subsubsection{Inductive structure on the motives~$\Z^{(n)}$}
		
			Let us recall from \S\ref{par: inductive} the morphisms~$i^{(n)}_\dR:\Z^{(n-1)}_\dR\rightarrow \Z^{(n)}_\dR$, which come from the identification~$X_{n-1}=\{x_n=1\}\subset X_n$. Let us consider an~$(n-1)$-form of the type
			$$\eta=\dfrac{P(x_1,\ldots,x_n)}{(1-x_1\cdots x_n)^N}\, dx_1\cdots dx_{n-1}\ ,$$
			with~$P(x_1,\ldots,x_n)$ a polynomial with rational coefficients and~$N\geq 0$ an integer. We say that such a form is integrable if the pullback~$\pi_n^*(\eta)$ does not have a pole along the exceptional divisor~$E_n$ ($\pi_n$ and~$E_n$ are introduced in \S\ref{par: def Z}). This can be characterized in the same way as in Propositions~\ref{prop: poles} and~\ref{prop: polesconverse}, but we will not need such a characterization. If~$\eta$ is integrable then its derivative~$d\eta$ is integrable in the sense of Definition~\ref{def: integrable}, and the restriction~$\eta_{|x_n=1}$, viewed as a form on~$X_{n-1}$, is also integrable. We then have classes~$[d\eta]\in\Z^{(n)}_\dR$ and~$[\eta_{|x_n=1}]\in\Z^{(n-1)}_\dR$. They are related by the formula
			$$i_{n,\dR}([\eta_{|x_n=1}])\equiv (-1)^{n-1}[d\eta] \;\;(\mathrm{mod}\;W_0\Z^{(n)}_\dR)\ ,$$
			which is proved as in the proof of Proposition~\ref{propdRrecurrence}, by noticing that~$\eta_{|x_n=0}$ is a polynomial, hence has weight zero. This formula is the de Rham-theoretic incarnation of Stokes's formula
			$$(-1)^{n-1}\int_{[0,1]^n}d\eta = \left( \int_{[0,1]^{n-1}}\eta_{|x_n=1} - \int_{[0,1]^{n-1}}\eta_{|x_n=0} \right) \equiv \int_{[0,1]^{n-1}}\eta_{|x_n=1} \;\;(\mathrm{mod}\;\mathbb{Q})\ .$$
			
			If we now choose to make the identification~$X_{n-1}=\{x_j=1\}\subset X_n$, for some index~$j=1,\ldots,n$, then we get a morphism~$i^{(n),j}_\dR:\Z^{(n-1)}_\dR\rightarrow \Z^{(n)}_\dR$, such that~$i^{(n)}_\dR=i^{(n),n}_\dR$. They satisfy the equation
			\begin{equation}\label{eq:appendixiota}
			i^{(n),j}_\dR([\eta_{|x_j=1}])\equiv (-1)^{j-1}[d\eta] \;\;(\mathrm{mod}\;W_0\Z^{(n)}_\dR)\ ,
			\end{equation}
			for~$\eta$ an integrable~$(n-1)$-form of the type
			\begin{equation}\label{eq:appendixetatype}
			\dfrac{P(x_1,\ldots,x_n)}{(1-x_1\cdots x_n)^N}\, dx_1\cdots \widehat{dx_j}\cdots dx_n\ .
			\end{equation}
			
			One easily notes that the morphism~$i^{(n),j}_\dR$ does not depend on the index~$j$, for instance by proving that Proposition~\ref{propdRrecurrence} is valid for any choice of~$j$: for every~$d=0,2,\ldots,n-1$, the map~$i^{(n),j}_\dR$ sends the class~$[\omega^{(n-1)}_d]$ to the class~$[\omega_d^{(n)}]$. We nevertheless keep the notation~$i^{(n),j}_\dR$ since these morphisms have different geometric interpretations.
		
		\subsubsection{Compatibility of~$\Phi_n$ with the induction}
		
			The crucial point is that the morphisms~$\Phi_n$ are compatible with the inductive structure (\ref{eq:appendixiota}) on the motives~$\Z^{(n)}_\dR$, in the sense of the following lemma.		
		
			\begin{lem}\label{lem:appendixPhi}
			For every~$j=1,\ldots,n$ and every differential~$(n-1)$-form~$\eta$ of type (\ref{eq:appendixetatype}) we have
			$$\Phi_n(d\eta)\equiv (-1)^{j-1}\Phi_{n-1}(\eta_{|x_j=1}) \;\;(\mathrm{mod}\;\Delta(V))\ .$$
			\end{lem}
			
			\begin{proof}
			We do the case~$j=n$, the general case being similar. It is enough to do the proof for a monomial
			$$\eta=\dfrac{x_1^{a_1-1}\cdots x_{n-1}^{a_{n-1}-1}x_n^{a_n}}{(1-x_1\cdots x_n)^N}\, dx_1\cdots dx_{n-1}\ ,$$
			with~$a_1,\ldots,a_{n-1}\geq 1$,~$a_n\geq 0$ and~$N\geq 1$. We have 
			$$(-1)^{n-1}d\eta=\left(a_n\dfrac{x_1^{a_1-1}\cdots x_n^{a_n-1}}{(1-x_1\cdots x_n)^N}+N\dfrac{x_1^{a_1}\cdots x_n^{a_n}}{(1-x_1\cdots x_n)^{N+1}}\right) dx_1\cdots dx_n\ ,$$
			and thus $(-1)^{n-1}\Phi_n(d\eta)$ equals
			$$a_n\binom{k+N-1}{N-1}\dfrac{1}{(k+a_1)\cdots (k+a_n)}+N\binom{k+N}{N}\dfrac{1}{(k+a_1+1)\cdots (k+a_n+1)}\cdot$$
			By writing~$\frac{a_n}{k+a_n}=1-\frac{k}{k+a_n}$ and~$N\binom{k+N}{N}=(k+1)\binom{k+N}{N-1}$, we get
			$$(-1)^{n-1}\Phi_n(d\eta)\equiv \binom{k+N-1}{N-1}\dfrac{1}{(k+a_1)\cdots (k+a_{n-1})} = \Phi_{n-1}(\eta_{|x_n=1}) \;\;(\mathrm{mod}\;\Delta(V))\ .$$
			\end{proof}
		
		\subsubsection{Proof of Theorem~\ref{thmappendix}}
		
			We prove Theorem~\ref{thmappendix} by induction on~$n$. The case~$n=1$ is trivial since in this case we have~$a_0(\omega)=b_0(\omega)=\int_0^1\omega$. Let us then assume that~$n\geq 2$ and that the theorem is proved for~$n-1$. Recall the notation 
			$$\omega_n^{(n)}=\dfrac{dx_1\cdots dx_n}{1-x_1\cdots x_n}$$
			for the representative of the highest weight basis element in~$\Z_{n,\dR}$; it satisfies~$\Phi_n(\omega_n^{(n)})=(k+1)^{-n}$. The short exact sequence (\ref{eq: short exact sequence Z}) implies that, for every~$\omega\in\Omega^{\mathrm{int}}_n$, we may write
			$$\omega=a_n(\omega)\omega_n^{(n)}+\sum_{j=1}^n d\eta_j$$
			with~$\eta_j$ an integrable~$(n-1)$-form of type (\ref{eq:appendixetatype}), for every~$j=1,\ldots,n$. The short exact sequence (\ref{eq: short exact sequence Z}) actually implies that in addition we can assume that the classes of~$d\eta_1,\ldots,d\eta_{n-1}$ are zero, but we will not need it here. By using (\ref{eq:appendixiota}) we may write
			$$[\omega]=a_n(\omega)[\omega_n^{(n)}] + \sum_{j=1}^n (-1)^{j-1}i^{(n),j}_\dR([(\eta_j)_{|x_j=1}]) \; \; (\mathrm{mod}\; W_0\Z^{(n)}_\dR)\ .$$
			Now Lemma~\ref{lem:appendixPhi} implies the formula
			$$\Phi_n(\omega)\equiv\dfrac{a_n(\omega)}{(k+1)^n} + \sum_{j=1}^n (-1)^{j-1}\Phi_{n-1}((\eta_j)_{|x_j=1}) \; \; (\mathrm{mod}\; \Delta(V)) \ .$$
			
			By using the induction hypothesis on the forms~$(\eta_j)_{|x_j=1}$ and the fact that the morphisms~$i^{(n),j}_\dR$ are compatible with the bases, this implies that we have
			$$\Phi_n(\omega) \equiv \sum_{r=2}^n \dfrac{a_r(\omega)}{(k+1)^r} \; \; (\mathrm{mod}\; \Delta(V)) \ ,$$
			which completes the proof. \\
		
			We note that a restatement of Theorem~\ref{thmappendix} is that the morphisms~$\Phi_n$ induce an isomorphism of graded vector spaces
			$$\Phi: \Z_\dR/W_0\Z_\dR \stackrel{\simeq}{\longrightarrow} (V/\Delta(V))_{\geq 2}\ ,$$
			where~$(V/\Delta(V))_{\geq 2}$ is the subspace of~$V/\Delta(V)$ characterized by the condition~$\beta_1=0$ and is graded by the morphisms~$\beta_n$,~$n\geq 2$.

\bibliographystyle{alpha}
\bibliography{biblio}

\end{document}